\documentclass[]{interact}

\usepackage{epstopdf}
\usepackage[caption=false]{subfig}

\usepackage[numbers,sort&compress]{natbib}
\bibpunct[, ]{[}{]}{,}{n}{,}{,}
\renewcommand\bibfont{\fontsize{10}{12}\selectfont}

\theoremstyle{plain}
\newtheorem{theorem}{Theorem}[section]
\newtheorem{lemma}[theorem]{Lemma}

\theoremstyle{definition}
\newtheorem{definition}[theorem]{Definition}
\newtheorem{example}[theorem]{Example}

\theoremstyle{remark}
\newtheorem{remark}{Remark}

\DeclareMathOperator*{\argmin}{arg\,min}

\usepackage{color}

\begin{document}


\title{Convergence analysis of stochastic higher-order majorization-minimization algorithms}

\author{
	\name{Daniela Lupu\textsuperscript{*} and Ion Necoara\textsuperscript{*,**}\thanks{Corresponding author: I. Necoara  (email:  ion.necoara@upb.ro)}}
	\affil{\textsuperscript{*}Automatic Control and Systems Engineering Department, University Politehnica Bucharest, 060042 Bucharest, Romania;   \textsuperscript{**}Gheorghe Mihoc-Caius Iacob Institute of Mathematical Statistics and Applied Mathematics of the  Romanian Academy, 050711 Bucharest, Romania.}}

\maketitle

\begin{abstract}
	Majorization-minimization schemes are a broad class of iterative methods targeting general optimization problems, including nonconvex, nonsmooth and stochastic. These algorithms minimize successively a sequence of upper bounds of the objective function so that along the iterations the objective value decreases. We present a stochastic higher-order algorithmic framework for minimizing the average of a very large number of sufficiently smooth functions.  Our stochastic framework is based on the notion of stochastic higher-order upper bound approximations of the finite-sum objective function and minibatching. We derive convergence results for nonconvex and convex optimization problems when the higher-order  approximation of the objective function yields an error that is $p$ times differentiable and has  Lipschitz continuous $p$ derivative.  More precisely, for general nonconvex problems we present asymptotic stationary point guarantees and under Kurdyka-Lojasiewicz property  we derive local  convergence rates ranging from sublinear to linear.  For convex problems with uniformly convex objective function we derive local (super)linear convergence results for our algorithm.  Compared to existing stochastic (first-order) methods, our algorithm  adapts to the problem's curvature and allows using any batch size.  Preliminary numerical tests support the effectiveness  of our algorithmic framework.
\end{abstract}



\section{Introduction}
The empirical risk minimization (also called finite-sum) problems appear in various  applications such as statistics and  machine learning \cite{BatCur:18,GooBen:16} or  distributed control and signal processing \cite{HyvKar:01,NecNed:11}.  Usually, these problems are nonconvex and nonsmooth due to the presence of regularization terms and  constraints. Another difficulty when dealing with  these problems is the large number of terms in the finite-sum and   the large number of variables.  Hence, stochastic methods are the most appropriate to solve finite-sum optimization problems.  Among these methods, those based on stochastic first-order oracle are widely used,  e.g.,  stochastic  gradient descent  \cite{BatCur:18,MouBac:11,Mairal:15,NemJud:09,Nec:20}   or stochastic proximal point   \cite{Nec:20,RosVil:14}.  In these works,   for (strongly) convex objective functions and diminishing stepsizes, sublinear convergence results  were derived.  Although  stochastic first-order methods  work well in  applications, their convergence speed  is known to slow down close to saddle points or in ill-conditioned landscapes \cite{Nes:13}. On the other hand, stochastic higher-order methods (including e.g., cubic Newton and third-order  methods) use the curvature of the objective function and thus they converge faster to a  (local) minimum \cite{AgaKam:20,LucKoh:19,KovRic:19}.

\medskip 

\noindent  \textbf{Related work}.  Most of the optimization methods  can be interpreted from the majorization-minimization point of view, which consist  of successively minimizing some upper bounds of the objective function \cite{Mairal:15,NecLup:20}.  More specifically, in these methods at each iteration one constructs and optimizes a local upper bound  model (majorizer) of the objective  function using first- or higher-order derivatives (e.g., a Taylor approximation with an additional  penalty term that depends on how well the model approximates the real objective).  In this work, we also focus on stochastic  high-order majorization-minimization  methods to optimize a finite-sum objective function. Deterministic higher-order majorization-minimization methods which construct a local Taylor model of the objective in each iteration with an additional regularization term has gained a lot of attention in the last decade due to their fast convergence \cite{BirGar:17,cartis2017,Nes:19,NesBor:06,Nes:Inexat19,NecLup:20}.  For example,  first-order methods achieve convergence rates of order $\mathcal{O}(1/k)$ for smooth convex optimization \cite{Nes:13}, while   higher-order methods of order $p >1$ have converge rates $\mathcal{O}(1/k^p)$  \cite{Nes:19,Nes:Inexat19,NecLup:20}, where $k$ is the iteration counter. Accelerated variants of higher-order methods  were also developed recently, see e.g., \cite{Gas:19,Nes:Inexat19}.  Yet, these better convergence results assume a deterministic setting where access to exact evaluations of the higher-order derivatives is needed and they do not directly translate to the stochastic case. 

\medskip 

\noindent To the best of our knowledge,  the use of high-order  derivatives in a stochastic setting has received little consideration in the literature so far.  For example, \cite{AgaKam:20,LucKoh:19} (see also \cite{BolByr:18,TriSte:18,ZhoXu:18})  present  stochastic optimization methods that use  third- (second-) order Taylor approximations with a  fourth- (cubic-) order regularizations to find local minima of smooth and  nonconvex finite-sum objective functions.  Usually, these methods use sub-sampled derivatives instead of computing them exactly  and are able to find fast (first-/second-order) critical points. However, these methods require large sub-samples,  e.g., the batch size must be proportional  to $\mathcal{O}({1/\epsilon^2})$, where $\epsilon$ is the  desired accuracy for solving the problem.  In general, the theoretical bounds for the batch size usually exceed the number of functions from the finite-sum, leading basically to   deterministic algorithms. Furthermore, quasi-Newton type methods that perform updates in a deterministic (cyclic) fashion are given e.g.,  in \cite{MokEis:18}.  {The papers most related to our work are  \cite{KovRic:19,Mairal:15}, where  stochastic variants of the majorization-minimization approach are derived based on  first- or second-order Taylor approximations with  a proper regularization term. For these methods (local)  linear rates are derived when the finite-sum objective function is  sufficiently smooth and  strongly convex.  \textit{However, there is no  complete convergence analysis for general  stochastic higher-order majorization-minimization algorithms for solving general finite-sum problems (including e.g., nonconvex problems or problems having  nonsmooth regularization terms).}} 

\medskip 

\noindent \textbf{Contributions}.  This paper presents a stochastic higher-order algorithmic framework for minimizing finite-sum (possibly nonconvex and nonsmooth) optimization problems.  Our  framework is based on stochastic higher-order upper bound approximations of the  (non)convex and/or (non)smooth finite-sum objective function, leading to a   minibatch \textit{stochastic  higher-order majorization-minimization}  algorithm, which we call SHOM. We derive convergence guarantees for the SHOM algorithm for  general  optimization problems when the upper bounds approximate the finite-sum objective function  up to an error that is $p \geq 1$ times  differentiable and has a Lipschitz continuous $p $ derivative; we call such upper bounds \textit{stochastic higher-order surrogate} functions.   The main challenge in analyzing convergence of SHOM, especially in the nonconvex setting,  is the fact that even in expectation the cost cannot be used as a Lyapunov function, unless a proper secondary sequence is constructed from the algorithm. For general nonconvex problems we prove that SHOM is a descent method {in expectation}, derive asymptotic stationary point guarantees and under Kurdyka-Lojasiewicz (KL) property  we  establish the first {local} linear or sublinear  convergence rates (depending on the KL parameter)  for stochastic higher-order type algorithms under such an assumption. For convex problems we derive local superlinear or linear  convergence results, provided that  the   objective function is uniformly convex. 

\medskip

\noindent To prove convergence, stochastic  higher-order methods usually require  the evaluation of a large number of  higher derivatives (e.g., gradients, Hessians, etc.). On the other hand, our algorithmic framework does not have this drawback. In particular, one variant of our method requires  at each iteration  the random selection of a single function from the finite-sum and the computation of only its  higher-order derivatives.  Moreover, unlike most existing stochastic higher-order methods (e.g.,  second- and third-order methods), ours  have faster local convergence  than  first-order methods, as they adapt to the objective function’s curvature. Besides providing a general framework for the design and analysis of stochastic higher-order  methods, in special cases, where complexity bounds are known for some particular algorithms, our convergence results recover the  existing bounds (see Remark \ref{rem:comp}). In particular, for strongly convex functions and  $p=2$ we recover the local linear rate from  \cite{KovRic:19}. For the deterministic  case, i.e.,  the batch size is equal to the number of functions in the finite-sum, we obtain  local superlinear convergence  as in \cite{DoiNes:2019}.  {Finally, for $p=1$ SHOM coincides with the popular MISO algorithm \cite{Mairal:15} and a byproduct of our convergence analysis leads to new convergence results for  MISO in the nonsmooth and nonconvex settings under the KL assumption.}  Numerical simulations also confirm the efficiency of our algorithm,  i.e.  increasing the approximation order can have beneficial effects in terms of convergence. 

\medskip 

\noindent \textbf{Content}. The paper is organized as  follows. In section 2 we introduce  notations and some generalities. Then,  in section 3 we formulate the  optimization problem and  our  algorithm. In section 4 we provide a convergence analysis for the nonconvex case. For convex problems we derive  convergence rates in section 5.  Finally, section 6 presents some  preliminary numerical experiments on  convex and nonconvex applications.


\section{Notations and generalities} 
{ We denote a finite-dimensional real vector space with $E$ and by $E^*$ its dual space composed by linear functions on $E$. For	such a function $s \in  E^*$, we denoted by $\langle s, x\rangle$ its value at $x \in E.$ Utilizing a self-adjoint positive-definite operator $B: E \rightarrow E^*$ (notation $B = B^{*} \succ 0$), we can endow  these spaces with\textit{ conjugate Euclidean norms}:
$$\|x\|=\langle B x, x\rangle^{1 / 2} \quad \forall x \in E,  \quad \|y\|_{*}=\left\langle y, B^{-1} y\right\rangle^{1 / 2}  \quad \forall y \in E^{*}.$$ 
\noindent  Under these settings, we have the folowing expression for the dual norm:  $\|y\|_{*}= \max_{x: \|x\| \leq 1} \langle y, x \rangle$. }   For a smooth function  $\psi:   \text{dom} \psi \rightarrow \mathbb{R}$,  $p\geq 1$ times continuously differentiable on the convex and open domain $\text{dom}  \psi \subseteq E$,  denote by $\nabla \psi(x)$  its gradient and  $\nabla^2 \psi(x)$ its Hessian   at the point $x \in \text{dom} \psi$. Note that  $\nabla \psi(x) \in E^* $  and $  \nabla^2 \psi(x) h \in E^*$ for all $h \in E$.  In what follows, we often work with directional derivatives:   $D^p \psi(x) \left[h_{1}, \ldots, h_{p}\right]$ denotes	  the $p$ directional derivative of function $\psi$ at $x \in \text{dom} \psi$ along directions  $h_1, \cdots, h_p \in E$ (see also \cite{Nes:19} for a similar exposition). Note that $ D^p \psi(x) [\cdot]$ is  a symmetric  $p$  multilinear form on $E$.  For example, for a smooth function $\psi$ one has for any $x \in \text{dom}  \psi $ and $h, \bar h \in E$ that $D \psi(x)[h] = \langle  \nabla \psi(x), h\rangle$ and $D^2 \psi(x)[h,\bar h] = \langle  \nabla^2 \psi(x) h, \bar h\rangle$. On the other hand, in this paper we consider  for a nonsmooth function $\psi$  the directional derivative  classically defined as: $D \psi(x)[h] = \lim_{t \downarrow 0} \left(  \psi(x + th) - \psi(x) \right)/t$ (note that this definiton  is valid  not only for convex functions and, obviously, for  differentiable functions, but also for a broader class of functions, so-called locally convex).  The abbreviation $D^p \psi(x)[h]^{p}$
is used when all directions are the same, i.e., $h_{1}=\cdots=h_{p}=h$ for some $h \in E$. The norm of $D^p \psi(x)$ is defined in the standard way (see \cite{Nes:19}):
\begin{equation*} 
	\|D^p \psi(x) \|:=\max _{\left\|h_{1}\right\|,\cdots,\left\|h_{p}\right\| \leq 1}\left|D^p \psi(x)\left[h_{1}, \ldots, h_{p}\right]\right|  =  \max _{\left\|h\right\| \leq 1}\left|D^p \psi(x)\left[h\right]^p\right|   \text{.}
\end{equation*}
Note that for any fixed $x, y \in \text{dom} \psi$ the form  $D^p \psi(x) [\cdot] - D^p \psi(y) [\cdot]$ is also $p$ multilinear and symmetric. Then, we define the following class of smooth functions: 

\begin{definition} Let $\psi: E \rightarrow \mathbb{R}$ be $p\geq 1$ times continuously differentiable. Then, the $p$ derivative of $\psi$ is Lipschitz continuous if there exist $L_p^{\psi} > 0$  for which the following relation holds:
	\begin{equation} \label{eq:1}
		\| D^p \psi(x) - D^p \psi(y) \| \leq L_p^{\psi} \| x-y \| \quad  \forall x,y \in \text{dom} \psi.
	\end{equation}
\end{definition}

\noindent We denote the Taylor approximation of $\psi$ around $x \in \text{dom} \psi$ of order $p$ by:
$$
T_p^{\psi}(y;x)= \psi(x) + \sum_{i=1}^{p} \frac{1}{i !} D^{i} \psi(x)[y-x]^{i}  \quad \forall y \in E.
$$		
\noindent It is known that if \eqref{eq:1} holds, then by the standard integration arguments  the residual between function value and its Taylor approximation can be bounded \cite{Nes:19}:
\begin{equation}\label{eq:TayAppBound}
	|\psi(y) - T_p^{\psi}(y;x) | \leq  \frac{L_p^{\psi}}{(p+1)!} \|y-x\|^{p+1}  \quad  \forall x,y \in \text{dom} \psi. 
\end{equation}
{Applying the same reasoning for the functions $\langle \nabla \psi(x), h\rangle$ and $\langle \nabla^2 \psi(x) h, h\rangle$, with direction $h \in E$ being fixed,   we also get the following inequalities valid for all $ x,y \in \text{dom} \psi$ and $p \geq 2$, see \cite{Nes:19}:}
\begin{align} \label{eq:TayAppG1}
	&\| \nabla \psi(y) - \nabla T_p^{\psi}(y;x) \|_* \leq \frac{L_p^{\psi}}{p!} \|y-x \|^p, \\
	\label{eq:TayAppG2}
	&\|\nabla^2 \psi(y) - \nabla^2 T_p^{\psi}(y;x) \| \leq \frac{L_p^{\psi}}{(p-1)!} \| y-x\|^{p-1}.
\end{align}

\noindent  {For the Hessian we consider  the spectral norm of self-adjoint linear operators (maximal module of all eigenvalues computed with respect to operator $B$).}  Below are some  examples of functions  having known Lipschitz continuous~derivatives.

\begin{example}\label{expl:1}
	Given $x_0 \in \mathbb{R}^n$ {and a matrix  $B \succ 0 $, defining the norm $\|x\|=\langle B x, x\rangle^{1 / 2} $ for all $ x \in \mathbb{R}^n$,}  then the power of the Euclidean norm $\psi(x)= \left\|x-x_{0}\right\|^{p+1}$ with $p \geq 1$ satisfies the Lipschitz continuous condition \eqref{eq:1} with $L_{p}^{\psi}=(p+1)!$ {(see Theorem 7.1 in  \cite{RodNes:20} for more details)}.  \qed
\end{example}

\begin{example}\label{expl:2}
	For given   $(a_{i})_{i=1}^m \in {E}^{*}$, consider the log-sum-exp function:
	\[
	\psi(x)=\log \left(\sum_{i=1}^{m} e^{\left\langle a_{i}, x\right\rangle}\right) \quad \forall x \in {E}.
	\]
	For $B:=\sum_{i=1}^{m} a_{i} a_{i}^{*}$ (assuming $B \succ 0$, otherwise we can reduce dimensionality of the problem) we define the norm $\|x\|=\langle B x, x\rangle^{1 / 2} $ for all $ x \in {E}$ and then the Lipschitz continuous condition \eqref{eq:1} holds for  $p=1, 2$ and $ 3$ with $L_{1}^{\psi}=1, L_{2}^{\psi}=2$ and $L_{3}^{\psi}=4$, respectively {(see Lemma 4 in \cite{DoiNes:20} for details)}. Note that for  $m=2$ and $a_1 =0$, we recover the logistic regression function, widely used in machine learning \cite{GooBen:16}. \qed
\end{example}

\begin{example}\label{expl:3} If the $p+1$ derivative of a function $\psi$ is bounded, then the $p$ derivative of $\psi$ is Lipschitz continuous. Moreover, any polynomial of degree $p$ (e.g., the $p$ Taylor approximation of $\psi$, denoted $T_p^\psi$),  has the $p$ derivative Lipschitz with the Lipschitz constant zero (in particular, $L_p^{T_p^\psi} = 0$). \qed 
\end{example}

\noindent Further, let us introduce the class of uniformly convex functions.

\begin{definition}
	\noindent A function $\psi(\cdot)$ is \textit{uniformly convex} of degree $q \geq 2$ with the uniform constant  $\sigma_{q}>0$, if for any 	$x, y \in \text{dom} \psi $ we have:
	
\vspace{-0.5cm}	

	\begin{equation}\label{eq:unifConv}
		\psi(y) \geq \psi(x)+\left\langle \psi_{x}, y-x\right\rangle+\frac{\sigma_{q}}{q}\|x-y\|^{q},
	\end{equation}

\vspace{-0.3cm}	

\noindent 	where  {$\psi_{x} \in E^*$} is an arbitrary vector (subgradient) from the subdifferential $\partial \psi(x)$. 
\end{definition}

\noindent Note that  $q = 2$ corresponds to the strongly convex functions. Next, we provide a simple example of a uniformly convex function (see e.g.,  \cite{Nes:Inexat19} for more details). 
\begin{example}\label{expl:4}
	Given $x_0 \in \mathbb{R}^n$, {$B = I_n$} and $q\geq 2$, then the power of the Euclidean norm $\psi(x) =\frac{1}{q} \| x-x_0\|^{q} $ is uniformly convex of degree $q$ with  $\sigma_{q} = \left( \frac{1}{2}\right)^{q-2}$.  \qed
\end{example}


\section{Stochastic higher-order majorization-minimization algorithms}
We consider  the finite-sum (possibly nonsmooth and nonconvex)  problem:
\vspace*{-0.4cm}
\begin{equation} 
	\label{fw:eq1}
	\min_{x \in \mathcal{X}} f(x) = \frac{1}{N} \sum_{i = 1}^{N} f_i(x), 
\end{equation}

\vspace{-0.3cm}

\noindent where each function $ f_i \!:\! \text{dom} f \! \rightarrow \mathbb{R}$,  having $\text{dom} f \! \subseteq\!  E$  open set, {is lower semi-continuous} and $\mathcal{X} \!\subseteq\! \text{dom} f$ is a given  closed convex set. Our approach consists of associating to $f$ a stochastic $p$ higher-order surrogate function $g$. Let us define this~notion:

\begin{definition} \label{def:1}
	Given $f$ as in \eqref{fw:eq1} and the set of  points $\{x_1,\cdots,x_N\} \subset \mathcal{X} $,  we denote $\hat{x} = \left[x_1^T \cdots x_N^T\right]^T$. Then,   the function $g(y; \hat{x}) = \frac{1}{N}\sum_{i=1}^{N} g_i(y;x_i)$, with  $ g_i(\cdot;x_i): \text{dom} g_i \rightarrow \mathbb{R}$  having  $\mathcal{X} \subseteq \text{dom} g_i$ , is called a stochastic $p$ higher-order surrogate function at $\hat{x}$ over $\mathcal{X}$ if the following relations hold for all $i=1:N$:
	\renewcommand{\theenumi}{\roman{enumi}}
	\begin{enumerate}
		\item each $g_i$ is a majorizer for the function $f_i$ on  $\mathcal{X}$, i.e.  $g_i(y;x_i) \geq f_i(y) \; \forall y \in \mathcal{X}$.
		\item The error function $h_i(y;x_i) = g_i(y;x_i) - f_i(y)$ {with $\text{dom} h_i = E$} has the $p$ derivative  Lipschitz continuous  with Lipschitz constant $L_p^{h_i}$.
		\item  \!\!\!\! The  $j$  derivatives satisfy  $\!\nabla^j h_i(x_i;x_i) \!=\! 0 \;\;  \forall\! j\!=\!0\!:\!p$, where $j\!=\!0$ means \!\! $h_i(x_i;x_i) \!=\!~0$. 
	\end{enumerate}
\end{definition}

\noindent Hence, the previous three properties yield the following relations for the finite-sum  objective function $f$:
\begin{enumerate}
	\item $g$ is a majorizer for the finite-sum objective function $f$, i.e.   $g(y,\hat{x})\geq f(y) \,\, \forall y\in \mathcal{X}$.
	\item The error function  $h(y;\hat{x}) = \frac{1}{N} \sum_{i=1}^{N} h_i(y;x_i)$ has the $p$ derivative smooth with Lipschitz constant $L_p^h = \frac{1}{N}  \sum_{i=1}^{N}L_p^{h_i}$.
\end{enumerate}

\noindent See also \cite{Mairal:15} for a similar definition of a stochastic  first-order surrogate function (i.e., for the particular case $p=1$) and \cite{NecLup:20} for the definition of a higher-order surrogate function in the  deterministic settings.  Note that our formulation of the surrogate extendes both definitions given in  \cite{Mairal:15} and in \cite{NecLup:20}.  Next, we give several nontrivial examples of stochastic higher-order  surrogate functions satisfying  requirements of Definition \ref{def:1}. 

\begin{example}
	\label{expl:5} 
	(\textit{Lipschitz $p$ derivative stochastic surrogate}).  
	Assume that  each function $f_i: E \rightarrow \mathbb{R}$ has the $p$ derivative smooth with the Lipschitz  constant $L_p^{f_i}$. Then,  we can define the following stochastic $p$ higher-order surrogate over $E$:
	$$
	g(y;\hat{x}) = \frac{1}{N}\sum_{i=1}^{N} \Big[  \underbrace{T_p^{f_i}(y;x_i) + \frac{M_p^i}{(p+1)!} \| y-x_i \|^{p+1}}_{=g_i(y;x_i)} \Big ],
	$$
	where $ M_p^i \geq  L_p^{f_i}$ and $x_i \in \mathcal{X}$ for all $i=1:N$. Then, the error function $h(y;\hat{x}) = g(y;\hat{x}) -f(y)$ has the $p$ derivative smooth with the  Lipschitz constant $L_p^h = \frac{1}{N}  \sum_{i=1}^{N}M_p^i + L_p^{f_i}$.  Let us note that when each  function $f_i$ is convex the stochastic $p$ higher-order surrogate function $g$ is also convex in the first argument for any $\hat{x}$, provided that $M_p^i \geq p L_p^{f_i}$. This is a consequence  of the following lemma.   \qed
\end{example}

\begin{lemma} \cite{Nes:19} 
	\label{lm:nes}
	Let $f$ be convex function with  the $p>2$ derivative Lipschitz continuous with constant $L_p^f$. Then,  for $M_p \geq pL_p^f$ and any $x \in {E}$ the  function: 
	$$
	g(y;x) = T_p^f(y;x) + \frac{M_p}{(p+1)!} \| y-x \|^{p+1}
	$$ 
	is convex in the first argument.  
\end{lemma}

\begin{example} \label{expl:6}
(\textit{composite functions}). Assume that $f(x) = \frac{1}{N} \sum_{i}^{N} \phi_i(x) +  \psi(x)$, where each (possibly nonconvex) function $ \phi_i$ has the $p$ derivative smooth  with the  Lipschitz  constant $L_p^{ \phi_i}$ and $\psi$ is a simple function (possibly nonsmooth and nonconvex). Assume also that $\mathcal{X} \subseteq \text{dom} \psi$ and $\psi$ is  proper lower-semicontinous function. Define $f_i =  \phi_i +  \psi$.  Then, we have the following stochastic $p$ higher-order surrogate over $\mathcal{X} $:
	$$
	g(y;\hat{x}) =\frac{1}{N}\sum_{i=1}^{N} \Big[  \underbrace{T_p^{ \phi_i}(y;x_i) +\frac{M_p^i}{(p+1)!} \|y -x_i\|^{p+1} + \psi(y) }_{=g_i(y;x_i)} \Big],
	$$	
where $M_p^i \geq  L_p^{ \phi_i}$. Further, the error function $h(y;\hat{x}) = g(y;\hat{x}) -f(y)$ has the $p$ derivative smooth with  the Lipschitz constant  $L_p^h =\frac{1}{N}  \sum_{i=1}^N M_p^i + L_p^{\phi_i}$.  From Lemma 	\ref{lm:nes} it follows that  the surrogate $g$ is  convex in the first argument provided that the functions $\phi_i$'s and $\psi$ are convex and the constants $M_p^i$'s are sufficiently large.  \qed
\end{example}

\begin{remark}
In  Examples \ref{expl:5} and \ref{expl:6}, 	if additionally, each function $f_i$ and $\phi_i$ satisfies the condition $f_i(y) \geq T_p^{f_i}(y;x_i)$ and  $\phi_i(y) \geq T_p^{\phi_i}(y;x_i)$ for all $y$ and $x_i$, respectively,  then the Lipschitz constant for the error function $h$ can be improved to $L_p^h = \frac{1}{N}  \sum_{i=1}^N M_p^i$, see  \cite{NecLup:20}. \qed
\end{remark}

\begin{example}\label{expl:7}
	(\textit{bounded  derivative functions}). Assume $p$ odd and each function $f_i$ has the $p+1$ derivative  bounded by a symmetric multilinear form $H_i$,  i.e $\langle D^{p+1}f_i(x)[h]^p, h\rangle \leq \langle H_i [h]^p, h\rangle$ for all $h \in \mathbb{E}$ and $x \in \mathcal{X}$.  Then, each $f_i$ has the $p$ derivative  Lipschitz with constant $L_p^{f_i} \leq \|H_i\|$ (see Example \ref{expl:3}) and  for any constants  $M_p^i \geq 1$ we have the following stochastic $p$ higher-order surrogate over $\mathcal{X} $:
	$$
	g(y;\hat{x}) = \frac{1}{N}\sum_{i=1}^N    \Big[  \underbrace{  T_p^{f_i}(y;x_i) + \frac{M_p^i}{(p+1)!} \langle H_i[y-x_i]^p, y-x_i \rangle }_{=g_i(y;x_i)} \Big].
	$$ 
	Moreover,  the error function $h(y;\hat{x}) = g(y;\hat{x}) -f(y)$ has the $p$ derivative smooth with  the Lipschitz constant  $L_p^h =\frac{1}{N}  \sum_{i=1}^{N}\|H_i\| + L_p^{f_i}$.  From Lemma 	\ref{lm:nes} it follows that  the surrogate $g$ is  convex in the first argument provided that the functions $f_i$'s  are convex and the constants $M_p^i$'s are sufficiently large.  \qed 
\end{example}

\begin{example}
	\label{expl:8}
	(\textit{high-order proximal  functions}).  Let $f_i$, with $i=1:N$, be general proper lower-semicontinous functions (possibly nonsmooth and nonconvex). Then, for any $p \geq 1$ we have the following stochastic $p$ higher-order surrogate over $\mathcal{X}$:
		$$
		g(y;\hat{x}) = \frac{1}{N}\sum_{i=1}^N  \Big[  \underbrace{ f_i(y) + \frac{M_p^i}{(p+1)!} \|y -x_i\|^{p+1}  }_{=g_i(y;x_i)} \Big],
		$$ 
		for  parameters $M_p^i >0$.	Moreover,  the error function $h(y;\hat{x}) = g(y;\hat{x}) -f(y)$ has the $p$ derivative smooth with  the Lipschitz constant  $L_p^h =\frac{1}{N}  \sum_{i=1}^{N} M_p^i$ (according to Example \ref{expl:1}).  Note that if $f_i$'s are convex functions, then the corresponding surrogates $g_i$'s become uniformly convex;  if $f_i$'s are weakly convex functions and $p=1$, then  the surrogate functions $g_i$'s  become strongly convex for appropriate choices of $M_p^i$'s;  even for nonconvex functions $f_i$'s  and $p>1$, the $p+1$ regularization can have beneficial effects in  convexifying   $f_i$'s.  The reader can find other examples of stochastic surrogates depending on the application at hand.  \qed   
\end{example}

\noindent For solving the general (possibly nonsmooth and nonconvex) finite-sum   problem \eqref{fw:eq1}  we propose the following minibatch  Stochastic  Higher-Order Majorization-minimization   (SHOM) algorithm:

\begin{center}
	\framebox{
		\parbox{12.5 cm}{
			\begin{center}
				\textbf{ SHOM algorithm }
			\end{center}
			{ Given {$\tau \in [1, N]$, $p \geq 1$ and} $x_0 \in \mathcal{X}$,  define $x_0^j=x_0 \;\; \forall j=1:N$.\\
				For $k\geq 0$ do:}
			\begin{enumerate}
			\item[1.] Chose uniformly random a subset $S_k \subseteq \{1, \cdots,\, N\}$ of size $\tau$. 
			 \item[2.]  For each $i_k \in S_k$ compute  {$p$ higher-order surrogate} $g_{i_k}(y;x_k^{i_k})$ of $f_{i_k}$ near $x_k^{i_k}$,  otherwise keep the previous surrogates.
			 \item[3.]  Compute $x_{k+1}$ as a stationary point of subproblem: 
			 \vspace{-0.5cm}
			\begin{equation}
				\label{eq:subproblem}
				 \min_{y \in \mathcal{X}} g(y;\hat{x}_k) := \frac{1}{N} \sum_{j=1}^{N} g_j(y;x_k^j),
			\end{equation}
		
		 \vspace*{-0.5cm}
		
			where $\hat{x}_k = [x_k^j]_{j=1:N}$ and  $
			x_k^j = \begin{cases}
				x_k, \quad j \in S_k\\
				x_{k-1}^j, \quad j \notin S_k.
			\end{cases}
			$			
         \end{enumerate}
	}}
\end{center}

\noindent {Note that SHOM is a variance reduced method and for $p=1$ and $\tau=1$ it coincides with MISO algorithm given in  \cite{Mairal:15}.}    In our analysis below for the nonconvex case we consider  $x_{k+1}$ to be a stationary point of the subproblem  \eqref{eq:subproblem}  satisfying additionally the following descend property:
\begin{equation}
	\label{eq:descGhom}
	g(x_{k+1};\hat{x}_k) \leq g(x_k;\hat{x}_k) \quad \forall k \geq 0.
\end{equation}
Hence,   in the nonconvex case SHOM does not require the computation of a (local) minimum of  the surrogate function $g$ at $\hat{x}_k$ over the convex set $\mathcal{X}$, we just need $x_{k+1}$ to be a stationary point of \eqref{eq:subproblem} satisfying the descent \eqref{eq:descGhom}.  Moreover,  in the convex case this decrease always holds, since in this case  $x_{k+1}$ is computed as the global optimum of \eqref{eq:subproblem}. Let us recall the definition of finite-sum error function (see Definition \ref{def:1}):
\begin{equation*}
	h(y;\hat{x}_k) = \frac{1}{N} \sum_{j=1}^{N} h_j(y;x_k^j)=g(y; \hat{x}_k) - f(y),
\end{equation*}
where  $h_j(y;x^j_k) = g_j(y;x_k^j) - f_j(y)$.   Note that we can  update very efficiently the model approximation $g$ in  SHOM algorithm (see also Section \ref{sec:impl}):   
\begin{align}
\label{eq:model-aprox}	
	g(y;\hat{x}_k) 	&= g(y;\hat{x}_{k-1}) + \frac{\sum_{j \in S_k} g_j(y;{x}_k) - g_j(y;{x^j_{k-1}})}{N},
\end{align}
Moreover, since the  subset of indices  $S_k$ of size $\tau$ are chosen uniformly random, then  for scalars $\theta_i$, with $i=1:N$, the following  holds in expectation:  
\begin{equation}\label{eq:2}
{	\mathbb{E} \left[ \sum_{i \in S_k} \theta_i \right] = \frac{\tau}{N} \sum_{i=1}^{N}  \theta_i.}
\end{equation}
\noindent We denote the conditional expectation w.r.t. the filtration defined as the sigma algebra generated by the history of the random index set $S$, $\mathcal{F}_{k} = \sigma(\{S_0,\, \dots ,\, S_{k-1}\})$, by  $\mathbb{E}_k[\cdot] = \mathbb{E}[\cdot \mid\mathcal{F}_k]$.  Finally, we denote the directional derivative of a function $f$ at any $x \in \text{dom }f$ (assumed open set) along $y-x$ by $\nabla f\left(x\right)[y - x]$. Note that if  $f$ is differentiable, the directional derivative is $ \nabla f\left(x\right)[y - x]= \langle \nabla f(x), y - x \rangle$.


\section{Nonconvex convergence analysis}
In this section we assume a general nonconvex finite-sum (possibly nonsmooth) objective function  $f$ which has directional derivative at each point in its open domain $\text{dom }f$ (see Examples 	\ref{expl:5},  	\ref{expl:6}, 	\ref{expl:7} and	\ref{expl:8}). Under these general assumptions we define next the notion of stationary points for problem \eqref{fw:eq1} (see also \cite{Mairal:15} for a similar definition).

\begin{definition}\textit{(asymptotic stationary points).}  
	Assume that the   directional derivative of the function $f$ at any  $x \in \mathcal{X} \subseteq \text{dom }f$ along $y-x$ exists for all $y \in E$. Then, a sequence $\left(x_{k}\right)_{k\geq 0}$ satisfies the \textit{asymptotic stationary point} condition for the nonconvex problem \eqref{fw:eq1} if the following holds:
	\[
	\liminf _{k \rightarrow+\infty} \inf _{x \in \mathcal{X}} \frac{D f\left(x_{k}\right)[x-x_{k}]}{\left\|x-x_{k}\right\|} \geq 0.
	\]
\end{definition}

\noindent Now we are ready to analyze the convergence behavior of SHOM under these general settings (nonconvexity and nonsmoothness). First, we prove that the sequence generated by SHOM satisfies the asymptotic stationary point condition and the sequence  of function values decreases in expectation. Hence, SHOM algorithm is a descent method (in expectation). These results will help us later to derive explicit convergence rates for nonconvex and convex cases. 

\begin{lemma}
	\label{th:nonconv-gen}
	Assume that the objective function $f$ is nonconvex, bounded below and  has directional derivatives over  $\mathcal{X}$. Moreover, assume that  $f$ admits a stochastic   $p \geq 1$ higher-order surrogate function $g$ over  $ \mathcal{X}$  as given in Definition \ref{def:1} and the sequence $(x_k)_{k\geq 0}$ generated by SHOM is bounded. Then, the sequence of function values  in expectation $\left( \mathbb{E}[f(x_{k})]\right)_{k \geq 0}$ monotonically decreases, $(f(x_{k}))_{k \geq 0}$ converges  almost surely  and $\left(x_{k}\right)_{k \geq 0}$ satisfies the asymptotic stationary point condition with probability one.
\end{lemma}

\begin{proof}
	From  \eqref{eq:descGhom},  the definition of the surrogate function $g$ and the update rule of SHOM algorithm, we get:   
	\begin{align}
		\label{eq:nc1}
		g(x_{k+1};\hat{x}_k) & \leq g(x_k;\hat{x}_k) = g(x_k;\hat{x}_{k-1}) + \frac{\sum_{j \in S_k} g_j(x_k;{x}_k) - g_j(x_k;x^j_{k-1})}{N}. 
	\end{align}
	Further,  from the definition of the surrogate function, we have that for $\forall j \in S_k$: 
	$$g_j(x_k;x^j_k) =g_j(x_k;{x}_k)= f_j(x_k) \leq g_j(x_k;x^j_{k-1}),$$
	Moreover, for $j \notin S_k$, we have: 
	$$g_j(x_k;x^j_k) = g_j(x_k;x^j_{k-1}). $$ 
	Using these relations and  \eqref{eq:descGhom}, we obtain the following descent:
	\begin{align}
		\label{eq:decrease}
		f(x_{k+1}) \leq 	g(x_{k+1};\hat{x}_k)  \leq g(x_{k};\hat{x}_k)   \leq g(x_k;\hat{x}_{k-1}),
	\end{align}
	where the first inequality follows from the fact that $g$ majorizes $f$. Thus, the sequence $(g(x_{k+1};\hat{x}_k))_{k\geq 0}$ is monotonically decreasing and bounded below with probability $1$, since $g$ by definition is majorizing $f$, which by our assumption is bounded from below. Taking  expectation w.r.t. $\mathcal{F}_{k+1} $,  we  obtain that the sequence $(\mathbb{E}[g(x_{k+1};\hat{x}_k)])_{k\geq 0}$  converges.    Further, using basic properties of conditional expectation in \eqref{eq:decrease}, we have: 
	\begin{align}
		\label{eq:rel2}
		&\mathbb{E}[g(x_{k+1};\hat{x}_k)] \leq \mathbb{E}[g(x_{k};\hat{x}_{k})] = \mathbb{E}[\mathbb{E}_k \left[ g_i(x_{k};x_{k})  \right]] = \mathbb{E}[\mathbb{E}_k \left[ f_i(x_k)  \right]] = \mathbb{E}[f(x_k)]. 
	\end{align}  
	Taking  expectation w.r.t. the whole set of minibatches $\mathcal{F}_{k+1} $ in \eqref{eq:decrease} and combining with \eqref{eq:rel2}, we also get $ \mathbb{E}[f(x_{k+1})] \leq \mathbb{E}[f(x_{k})]$, i.e., SHOM algorithm is a descent method in expectation. Note also that the nonnegative quantity $\mathbb{E}\left[ g_{i}\left( x_{k+1};{x}_k\right) -  f_{i}\left(x_{k+1}\right)\right]$ is the summand of a converging sum and we have:
	$$
	\begin{aligned}
		&0 \leq  \mathbb{E}  \left[ \sum_{k=0}^{\infty} \sum_{i \in S_{k+1}}  g_{i}( x_{k+1};x_k) - f_{i} ( x_{k+1})\right] = \sum_{k=0}^{\infty} \mathbb{E}\!   \left[\sum_{i \in S_{k+1}}g_{i}\left( x_{k+1};x_k\right)-\! f_{i}\left(x_{k+1}\right)\right] \\
		&= \! \sum_{k=0}^{\infty} \mathbb{E}\left[\mathbb{E}_{k+1}\left[ \sum_{i \in S_{k+1}} g_{i}\left( x_{k+1};x_k\right)-\! f_{i}\left(x_{k+1}\right)\right]\right]\overset{\eqref{eq:2}}{=} \tau \sum_{k=0}^{\infty}\mathbb{E}\left[ g\left(x_{k+1};\hat{x}_k \right) -\! f(x_{k+1})\right],
	\end{aligned}
	$$
	where we used the Beppo-L\`{e}vy theorem to interchange the expectation and the sum in front of nonnegative quantities {(see also \cite{Mairal:15} for a similar derivation in the case $\tau=1$)}. Moreover, considering \eqref{eq:rel2}, we obtain:
	\begin{align*}
		\mathbb{E}\left[\sum_{k=0}^{+\infty} \sum_{i \in S_{k+1}}g_{i}\left( x_{k+1};x_k\right)-f_{i}\left(x_{k+1}\right)\right] 
		&\leq \tau \sum_{k=0}^{+\infty} \mathbb{E}\left[f(x_k) - f(x_{k+1})\right]\\ 
		&{\leq} \tau \left(f(x_0) - f^* \right) < \infty,
	\end{align*}
	where for the last inequality we used that  $f$ is assumed  lower bounded  e.g., by  $f^* > -\infty$. Hence,  nonnegative sequence $ (g(x_{k+1}; \hat{x}_{k}) \!-\! f(x_{k+1}))_{k \geq 0}$ converges to $0$ almost surely, and thus  sequence of function values  $(f(x_{k}))_{k \geq 0}$ also converges  almost surely.
	
	\medskip 
	
	\noindent  For the second part of the theorem,  we have from the optimality of $x_{k+1}$ that:
	$$
	\nabla g(x_{k+1};\hat{x}_{k})[y -x_{k+1}]  \geq 0 \quad \forall y \in \mathcal{X}.
	$$
	Thus, from the definition of the error function $h$ and the fact that $f$ is assumed to have directional derivatives, we have using basic calculus rules that:
	\begin{align*}
		\nabla f(x_{k+1})[y-x_{k+1}] &  = \nabla g(x_{k+1}; \hat{x}_{k})[y-x_{k+1}] - \nabla h(x_{k+1}; \hat{x}_{k})[y-x_{k+1}] \\
		& \geq - \langle  \nabla h(x_{k+1}; \hat{x}_{k}),  y- x_{k+1} \rangle \quad  \forall y \in \mathcal{X}. 
	\end{align*}
	Using the Cauchy-Schwarz inequality in the previous relation, we get:
	\begin{equation}\label{eq:25}
		\nabla f(x_{k+1})[y-x_{k+1}] \geq -\| \nabla h(x_{k+1}; \hat{x}_k)\|_* \, \|y- x_{k+1} \|\quad \forall y \in \mathcal{X}.
	\end{equation}
	For simplicity of the notation, let us define the function $\tilde{h}(y) = h(y;\hat{x}_k)$, which,  according to the Definition \ref{def:1}, has the $p$ derivative smooth with Lipschitz constant $L_p^h$.	Further, let us consider the following auxiliary  point:  
	$$y_{k+1} = \argmin_{y \in \mathbb{E}} T_p^{\tilde{h}}(y;x_{k+1}) + \frac{M_p^h}{(p+1)!} \|y-x_{k+1} \|^{p+1},$$
	where  $M_p^h > L_p^h$. Let us consider first the nontrivial case  $p \geq 2$. Using the (global) optimality of $y^{k+1}$, we have:
	\begin{align*}
		&T_p^{\tilde{h}}(y_{k+1};x_{k+1}) + \frac{M_p^h}{(p+1)!} \| y_{k+1}-x_{k+1}\|^{p+1} \\
		&  \leq T_p^{\tilde{h}}(x_{k+1};x_{k+1}) + \frac{M_p^h}{(p+1)!} \| x_{k+1}-x_{k+1}\|^{p+1} \\
		&= h(x_{k+1};\hat{x}_k) + \sum_{i=1}^{p} \frac{1}{i !} D^{i} h(x_{k+1};\hat{x}_k)[x_{k+1}-x_{k+1}]^{i} = h(x_{k+1};\hat{x}_k).
	\end{align*} 
	Moreover, developing the left term of the previous  inequality, we get:
	\begin{align*}
		& h(x_{k+1};\hat{x}_k) + \sum_{i=1}^{p} \frac{1}{i!} D^{i} h(x_{k+1};\hat{x}_k)[y_{k+1}-x_{k+1}]^{i} + \frac{M_p^h}{(p+1)!} \| y_{k+1}-x_{k+1}\|^{p+1} \\ 
		& \leq  h(x_{k+1};\hat{x}_k).
	\end{align*} 
	In conclusion, we obtain:
	\begin{equation}\label{eq:24}
		\sum_{i=1}^{p} \frac{1}{i!} D^{i} h(x_{k+1};\hat{x}_k)[y_{k+1}-x_{k+1}]^{i} + \frac{M_p^h}{(p+1)!} \| y_{k+1}-x_{k+1}\|^{p+1} \leq 0.
	\end{equation}
	From relation \eqref{eq:TayAppBound} we also have:
	\begin{equation*}
		\tilde{h}(y) =  h(y;\hat{x}_k) \leq T_p^{\tilde{h}}(y;x_{k+1}) + \frac{L_p^h}{(p+1)!} \| y-x_{k+1}\|^{p+1} \quad \forall y \in \mathbb{E}. 
	\end{equation*}
	We rewrite this relation for our chosen auxiliary point $y_{k+1}$ and get:
	\begin{align*}
		& h(y_{k+1};\hat{x}_k)  \leq T_p^{\tilde{h}}(y_{k+1};x_{k+1}) + \frac{L_p^h}{(p+1)!} \| y_{k+1}-x_{k+1}\|^{p+1} \\
		& = h(x_{k+1};\hat{x}_k) + \sum_{i=1}^{p} \frac{1}{i!} D^{i} h(x_{k+1};\hat{x}_k)[y_{k+1}-x_{k+1}]^{i} + \frac{L_p^h}{(p+1)!} \| y_{k+1}-x_{k+1}\|^{p+1} \\
		&\overset{\eqref{eq:24}}{\leq}h(x_{k+1};\hat{x}_k) - \frac{M_p^h -L_p^h}{(p+1)!} \| y_{k+1}-x_{k+1}\|^{p+1}.
	\end{align*}
	Recalling that $h$ is nonnegative, then it follows that:
	\begin{align*}
		0 \leq h(x_{k+1};\hat{x}_k) - \frac{M_p^h -L_p^h}{(p+1)!} \| y_{k+1}-x_{k+1}\|^{p+1}.
	\end{align*}
	This leads to:
	\begin{equation}
		\label{eq:lyap1}
		h(x_{k+1};\hat{x}_k) \geq  \frac{M_p^h -L_p^h}{(p+1)!} \| y_{k+1}-x_{k+1}\|^{p+1} \geq 0.
	\end{equation}		
	Since $h(x_{k+1}; \hat{x}_k)) = g(x_{k+1}; \hat{x}_{k}) - f(x_{k+1})$, using the first part of the proof we get that the sequence  $(h(x_{k+1}; \hat{x}_k))_{k \geq 0}$ converges to $0$ almost sure. Hence,  the sequence $(y_{k+1} - x_{k+1})_{k \geq 0}$ converges to $0$ almost sure. Moreover, from the optimality conditions for  $y_{k+1}$, we have:
	\begin{align}
		\label{eq:optcondy}
		\nabla h(x_{k+1};\hat{x}_{k}) + H_{k+1}[y_{k+1}-x_{k+1}] = 0,
	\end{align}
	where  we denote the matrix  
	\begin{align*}
		H_{k+1} & = \nabla^2 h(x_{k+1};\hat{x}_k) + \sum_{i=3}^{p} D^i  h(x_{k+1};\hat{x}_k)[y_{k+1}-x_{k+1}]^{i-2} + \frac{M_p}{p!} \| y_{k+1} -x_{k+1}\|^{p-1} B.
	\end{align*}
	Since the sequence $(x_{k})_{k\geq 0}$ generated by SHOM is assumed bounded and $h$ is $p$ times continuously  differentiable, then $ \nabla^{i} h(x_{k+1}; \hat{x}_k)$ are bounded for all $i= 2:p$. Moreover, since $y_{k+1} -x_{k+1}  \rightarrow 0$ as $k \to \infty$,  it follows that $y_{k+1}$ is also bounded (due to boundedness of $x_k$). Therefore, for $p\geq 2$ we have proved that   $H_{k+1}$ is bounded and  consequently:  
	$$ \nabla h(x_{k+1};\hat{x}_{k}) \to 0  \;\; \text{as}  \;\; k \to \infty \;\; \text{alsmost sure}. $$ 
	For  $p=1$ we can just take $y_{k+1} = x_{k+1} - 1/L_1^h \nabla h (x_{k+1;\hat{x}_k})$. {Then, using that $h(\cdot;\hat{x}_k)$ has gradient Lipschitz  with constant $L_1^h$, from \eqref{eq:TayAppBound} for $h$ and $p=1$, we obtain}:
	\[   0 \leq h(y_{k+1};\hat{x}_k)  \leq h(x_{k+1}; \hat{x}_k) - \frac{1}{2 L_1^h} \| \nabla h(x_{k+1};\hat{x}_k) \|^2_*, \]
	which further yields
	\[  \frac{1}{2 L_1^h} \| \nabla h(x_{k+1};\hat{x}_k) \|^2_*  \leq h(x_{k+1};\hat{x}_k) \to 0  \; \text{as}  \; k \to \infty, \]
	as  $h(x_{k+1};\hat{x}_k)$ converges to zero almost surely. Therefore, also in the case $p=1$ we have  $\nabla h(x_{k+1};\hat{x}_{k}) \to 0$.  Using this result in \eqref{eq:25}, minimizing over $y \in \mathcal{X}$ and taking the infimum limit,  we get that the sequence $\left(x_{k}\right)_{k \geq 0}$ satisfies the asymptotic stationary point condition for the nonconvex problem \eqref{fw:eq1}. This concludes our statements.  
\end{proof}

\noindent {Note that the main difficulty in the previous proof is to handle  $h$ having   $p \geq 2$ derivative smooth. We overcome  this difficulty by introducing a new sequence $(y_k)_{k \geq 0}$, proving that it has similar properties as the sequence  $(x_k)_{k \geq 0}$ generated by SHOM and then  using the  optimality conditions for $y_k$ instead of $x_k$.  Note that in practice we do not need to compute $y_k$.}  Moreover, if $f$ is differentiable and $\mathcal{X} = E$, the previous theorem states that the sequence of gradients $(\nabla f(x_{k}))_{k \geq 0}$ converges to $0$ almost sure.


	\subsection{Convergence rates under Kurdyka-Lojasiewicz  property}
	Recent works demonstrate that special geometric properties
	of the objective function, such as the Kurdyka-Lojasiewicz (KL) property, enable fast convergence rates for some first- or higher-order algorithms, see e.g., \cite{BolDan:07,NecLup:20}.  {However, from our knowledge, this is the first work deriving  convergence rates for  stochastic higher-order majorization-minimization algorithms under the KL property.  The main
challenge  when analyzing convergence of a stochastic algorithm under the KL is the fact that it is not trivial to combine  the local property of the KL with the stochasticity of the iterates. We overcome this problem using Egorov's theorem, which allows us to pass from almost sure convergence to uniform convergence.}  In this section we derive convergence rates for SHOM under the KL condition. For simplicity of the exposition, we assume  in this section $ \mathcal{X} = E$ and  the functions $f$ and  $g$ are  subdifferentiable on $E$.  Denote:

\vspace{-0.6cm}

	$$  S_f(x) = \text{dist}(0, \partial f(x)) \;  \left(:= \inf_{f^x \in \partial f(x)} \| f^x\|_* \right) \quad  \forall x \in E. $$  
	 
\vspace{-0.2cm}
	
\noindent If  $\partial f(x) = \emptyset$, we set $S_f(x) = \infty$.  We also assume that the optimization problem \eqref{fw:eq1} has  finite optimal value.   {Let $\Omega$ be the set of limit points of the sequence $(x_k)_{k \geq 0}$ generated by SHOM. From Lemma \ref{th:nonconv-gen}, we note that any limit point $x^* \in \Omega$ satisfies the first order stationary point condition for the original problem \eqref{fw:eq1}. Let us prove that  the objective function $f$ is constant on the set of limit points of the sequence $(x_k)_{k \geq 0}$, i.e., $\Omega$. Indeed, from Lemma \ref{th:nonconv-gen} we have that $\left(f(x_{k})\right)_{k \geq 0}$  converges {almost surely} to e.g., $ f_* > -\infty$}.  Let $x_*$ be a limit point of  $(x_k)_{k \geq 0}$.  Then, there exists a subsequence   $(x_{k_\ell})_{\ell \geq 0}$ such that $ x_{k_\ell} \to x_*$,  i.e., $x_* \in \Omega$. In this section only, we assume that $f$ is continuous.  
Then, we obtain $\lim_{\ell \to \infty}  f(x_{k_\ell}) = f(x_*) = f_*$ a.s., i.e., $f(\Omega) = f_*$   Based on these properties of $(x_k)_{k \geq 0}$, in the rest of this section we assume that $f$ satisfies additionally the following KL inequality on {a neighbourhood $\mathcal{U}$ of $\Omega$ and for some $\varepsilon >0$}, i.e.,  \cite{BolDan:07}:
	\begin{align}
		\label{eq:kl}
		f(x) - f_* \leq  \sigma_q  S_f(x)^q \quad \forall x \in \mathcal{U},  \; f_* < f(x) < f_*+\varepsilon, 
	\end{align} 

	\vspace{-0.2cm}
	
	\noindent where the KL parameter $q>1$ and $\sigma_q>0$.   Note that the KL property holds for a large class of functions including semi-algebraic functions and uniformly convex functions, see \cite{BolDan:07} for more examples. The next lemma  will be useful later in this section. Denote by  $\textbf{1}_{\Gamma}$  the indicator function of a set ${\Gamma}$.
	
	\begin{lemma}
		\label{lema:prob}	
		Given a \textcolor{black}{bounded} random variable $0 \leq  X \leq \Delta$ and a measurable set $\Gamma_\delta$ such that $\mathbb{P}(\Gamma_\delta) \geq 1-\delta$, the following inequality holds: $$ \mathbb{E}[X] - \Delta \sqrt{\delta} \leq   \mathbb{E}[X  \textbf{1}_{\Gamma_{ \delta}}]  \leq  \mathbb{E}[X]. $$
	\end{lemma}	
	
	\begin{proof}
		The  inequality from the right hand side is obvious, since $X \geq 0$.  For the other inequality, we can write $X = X \textbf{1}_{\Gamma_{ \delta}} + X \textbf{1}_{\Gamma_{ \delta}^C}$,  where  ${\Gamma_{ \delta}^C}$  is the complement of the set  ${\Gamma_{ \delta}}$. Then,  using the Cauchy-Schwartz inequality and the boundedness of $X$, we have: $\mathbb{E}[X] -  \mathbb{E}[X  \textbf{1}_{\Gamma_{ \delta}}] =   \mathbb{E}[X  \textbf{1}_{\Gamma_{ \delta}^C}]  \leq   ( \mathbb{E}[X^2] \;  \mathbb{E}[  \textbf{1}_{\Gamma_{ \delta}^C}^2]  )^{1/2}  \leq (\Delta^2   \mathbb{P}(\Gamma_{ \delta}^C))^{1/2}  \leq (\Delta^2 \delta)^{1/2}$.		
	\end{proof}		 
\noindent Note that since $f(x_{k})  \to f_*$ a.s., then there exists $\Delta < \infty$ such that $|f(x_{k})  - f_*| \leq \Delta$  for all $k \geq 0$. We get the following convergence rates for SHOM under the KL. 
		\vspace*{-0.1cm}
	\begin{theorem}
		\label{th:nonconv-gen-kl}
		Let  the objective function $f$ be proper and continuous (possibly  nonconvex),   satisfy the KL property \eqref{eq:kl} for some $q>1$ and $\varepsilon>0$,  admit   a stochastic   $p \geq 1$ higher-order surrogate  $g$  as in Definition \ref{def:1} and $M_p^h > L_p^h$.  \textcolor{black}{Additionally, assume that   the sequence  $(x_k)_{k \geq 0}$ generated by SHOM  satisfies $0 \leq f(x_{k})  - f_* \leq \Delta$   for all $k \geq 0$.}   Then,  for any $\delta>0$  exists  $k_{ \delta,\varepsilon}$ such that with probability at least $1 - \delta$:
	\vspace*{-0.1cm}
		\begin{enumerate}
			\item If $q = p+1$,  \textcolor{black}{then    locally $f(x_k)$ converges linearly  to a neighborhood of   $f_*$:}
			{$$ \mathbb{E}[f(x_{k})] - f_* \leq \left( \frac{\sigma_{q} c^q q!}{M_p^h - L_p^h + \sigma_{q} c^q q!}\right)^{k - k_{ \delta,\varepsilon}}  \left(\mathbb{E} [ f(x_{k_{ \delta,\varepsilon}}) ] -f_* \right)  {+ \Delta \sqrt{\delta}},$$} where $c = \max_{k \geq 0} \| H_{k+1}\|$ and $k > k_{ \delta,\varepsilon}$. 
			
			\item  If $q < p+1$,  \textcolor{black}{then    locally $f(x_k)$ converges sublinearly  to a neighborhood of   $f_*$:}
			{$$ \mathbb{E}[f(x_{k})] - f_* \leq \frac{ \mathbb{E}[f(x_{k_{ \delta,\varepsilon}})] - f_* }{ (1 + \bar \alpha (k -  k_{ \delta,\varepsilon}))^ \frac{q}{p+1-q}} {+ \Delta \sqrt{\delta}}, 
			$$}
	where $\bar \alpha >0$ is some appropriate constant and $k  > k_{ \delta,\varepsilon}$.
		\end{enumerate}	  
	\end{theorem} 

\begin{proof}
 From the proof of  Lemma \ref{th:nonconv-gen} (i.e., combining \eqref{eq:lyap1} with  	$h(x_{k+1}; \hat{x}_k)) = g(x_{k+1}; \hat{x}_{k}) - f(x_{k+1})$ and then taking expectation and the inequality \eqref{eq:rel2}), we have: 	
		\begin{align}
			\label{eq:hfy}
			{(M_p^h - L_p^h)}{((p+1)!)^{-1}} 	\mathbb{E} \left[ \| y_{k+1} - x_{k+1}\|^{p+1} \right ] \leq   	\mathbb{E} [f(x_k)] - 	\mathbb{E} [ f(x_{k+1}) ] \quad \forall k \geq 0.  	
		\end{align}	
Therefore, in expectation the cost function in combination with the sequences $x_k$ and $y_k$  can be used as a Lyapunov function.		From the optimality condition for $x_{k+1}$, we have $0 \in \partial g(x_{k+1};\hat{x}_k)$  and from $h= g-f$ we get  $- \nabla h(x_{k+1};\hat{x}_k)  \in  \partial f(x_{k+1})$  the  (limiting) subdifferential of $f$ at $x_{k+1}$.  Then, using 	\eqref{eq:optcondy}
		we obtain: 
		\begin{align} 
			\label{eq:nfyx} 
			S_f(x_{k+1})  & = \text{dist} (0, \partial f(x_{k+1}))   \leq   \| \nabla h(x_{k+1};\hat{x}_k) \|_* \nonumber \\
			& =  \| H_{k+1}(y_{k+1}-x_{k+1}) \|_*  \leq c \cdot \| y_{k+1}-x_{k+1} \|  \quad \forall k \geq 0, 
		\end{align}  
		where $c = \max_{k \geq 0} \| H_{k+1}\| < \infty$ according to  Lemma \ref{th:nonconv-gen}. Further, from Lemma \ref{th:nonconv-gen}  we have that $f(x_{k}) \to f_{*}$ a.s., which means that there exists some measurable set $\Gamma$ such that $ \mathbb{P}[\Gamma] =1$ and for  $\varepsilon$ and $\gamma \in \Gamma$  there exists $k_{\varepsilon}(\gamma)$ such that for any $k \geq k_{\varepsilon}(\gamma)$ we have   $f(x_k(\gamma)) - f_{*} \leq \sigma_q  S_f (x_k(\gamma))^q$.  Further,  invoking measure theoretic arguments to pass from almost sure convergence to almost uniform convergence thanks to Egorov’s theorem, we can prove that  for any $\delta>0$  there exists a measurable set $\Gamma_{ \delta}  \textcolor{black}{\subseteq \Gamma}$, such that $ \mathbb{P}[\Gamma_{\delta}]  \textcolor{black}{\geq} 1 - \delta$,  and for any  $\varepsilon >0$ there exists  $k_{\delta,\varepsilon} > 0$ such that for all $\gamma \in \Gamma_{ \delta}$ and $k \geq k_{\delta,\varepsilon} $ we have $f(x_k(\gamma)) - f_* \leq \sigma_q S_f(x_k(\gamma))^q$. Hence,  with probability at least  $1- \delta$ the sequence $(x_k)_{k\geq 0}$  satisfies KL  \textcolor{black}{on $\Gamma_{ \delta}$}  for all $k \geq k_{ \delta,\varepsilon} $, {i.e.,  $\left(f(x_k) - f_* \right) \textbf{1}_{\Gamma_{ \delta}} \leq \sigma_q S_f(x_k)^q \textbf{1}_{\Gamma_{ \delta}}$ for all $k \geq k_{ \delta,\varepsilon} $.  }  Then, with probability at least  $1- \delta$:
			\vspace*{-0.5cm}
		\begin{align}
			\label{eq:kl-rec}
			\left( f(x_{k+1})	- f_* \right)  {\textbf{1}_{\Gamma_{ \delta}}} & \leq \sigma_q   S_f(x_{k+1})^q {\textbf{1}_{\Gamma_{ \delta}}}  \leq \sigma_q   S_f(x_{k+1})^q  \overset{\eqref{eq:nfyx} }{\leq}  \sigma_q  c^q    \| y_{k+1}-x_{k+1}  \|^q  \nonumber \\ 
			 & =  \sigma_q  c^q   \left( \| y_{k+1}-x_{k+1}  \|^{p+1} \right)^{\frac{q}{p+1}} \quad \forall  k \geq k_{ \delta,\varepsilon}. 
		\end{align}	
		
		\vspace{-0.3cm}
		
	\noindent 	 \textcolor{black}{ Denote $\Delta_k = \mathbb{E} [  f(x_k)] - f_*$ and $C = \sigma_q  c^q  \left( \frac{(p+1)!}{M_p^h - L_p^h}  \right)^{\frac{q}{p+1}} $. If $q=p+1$, using Lemma 	\ref{lema:prob},  taking expectation in 	\eqref{eq:kl-rec} and combining with 	\eqref{eq:hfy}, we get:}
		\vspace*{-0.5cm}	
		\begin{align}
			\label{eq:kl-linear}
			\Delta_{k+1} - {\Delta \sqrt{\delta} } \leq C (\Delta_k - \Delta_{k+1})  \quad \forall k \geq k_{ \delta,\varepsilon},
	\end{align}
		which  \textcolor{black}{proves the first statement}. If  $q<p+1$,  using Lemma 	\ref{lema:prob},  taking expectation in 	\eqref{eq:kl-rec} and combining with 	\eqref{eq:hfy}, we obtain the following:
			\vspace*{-0.3cm}	
		\begin{align}
			\label{eq:kl-recc}
			& \mathbb{E} [	f(x_{k+1}) ]	- f_* - {\Delta \sqrt{\delta} }   \leq \sigma_q   \mathbb{E} [ S_f(x_{k+1})^q ]  \overset{\eqref{eq:nfyx} }{\leq}   \sigma_q  c^q    \mathbb{E} \left[ \left( \| y_{k+1}-x_{k+1}  \|^{p+1} \right)^{\frac{q}{p+1}} \right ]  \nonumber \\
			&  \leq  \sigma_q  c^q   \left( \mathbb{E} \left[  \| y_{k+1}-x_{k+1}  \|^{p+1} \right ]  \right)^{\frac{q}{p+1}}  \leq C   \left( 	\mathbb{E} [f(x_k)] - 	\mathbb{E} [ f(x_{k+1}) ] \right)^{\frac{q}{p+1}},  
		\end{align}	
		where in the third inequality we used concavity of  the function $u \mapsto (u)^{\frac{q}{p+1}}$ on $\mathbb{R}_+$ for $q<p+1$.  Hence, in this case  we obtain the following recurrence:
		\begin{align*}
			\Delta_{k+1} - \textcolor{black}{\Delta \sqrt{\delta} }  \leq C (\Delta_k - \Delta_{k+1})^{\frac{q}{p+1}}    \quad \forall k \geq k_{ \delta,\varepsilon}. 
		\end{align*}
		If we denote $ \Delta_k  - \textcolor{black}{\Delta \sqrt{\delta} } = C^{\frac{p+1}{p+1-q}} \tilde \Delta_k $, we further get the recurrence:
		\[  \tilde \Delta_{k+1}^{1 + \frac{p+1-q}{q}} \leq \tilde \Delta_k - \tilde  \Delta_{k+1}, \; \text{with} \;  \frac{p+1-q}{q} >0, \]
		and \textcolor{black}{then using similar arguments as in Lemma 11 in \cite{Nes:Inexat19}}, we obtain the following sublinear rate
		\begin{align}    
			\label{eq:kl-sublinear}
			\tilde \Delta_{k}  \leq \frac{   \tilde \Delta_{k_{ \delta,\varepsilon}} }{ (1 + \bar \alpha (k-k_{ \delta,\varepsilon}))^ \frac{q}{p+1-q}} \quad \forall k > k_{ \delta,\varepsilon}, 
		\end{align}
		for some appropriate  constant $\bar \alpha>0$, 	which  \textcolor{black}{proves the second statement}.
	\end{proof}	 

\noindent Previous theorem establishes linear or sublinear convergence results when the objective function satisfies the KL property without imposing convexity requirements on it.  For $p=1$ and $q=2$  we also get  new convergence rates for  MISO algorithm from \cite{Mairal:15} in the general nonsmooth and nonconvex settings. In the next section we derive convergence rates for convex objective functions.


\section{Convex  convergence analysis}
Throughout this section, we  assume that the objective functions $f$  is  convex and  $x_{k+1}$  is the global minimum of subproblem \eqref{eq:subproblem}. Note that if the functions $f_i$'s are all convex, then  for appropriate choices of constants  $M_p^i$'s (see Lemma \ref{lm:nes}), the subproblem \eqref{eq:subproblem} remains convex and then  the global minimum  $x_{k+1}$ can be computed using convex optimization algorithms. We analyze further the global convergence of the algorithm SHOM in the convex settings. Let $f^*$ and $x^*$ be the optimal value and a minimizer of the convex optimization problem 	\eqref{fw:eq1}.  As  usually assumed in convex optimization (see also e.g.,  \cite{Nes:19,Nes:13}), we consider that there exists $R>0$ such that:  
\begin{equation}
	\label{eq:d}
	\|x_k - x^* \| \leq R  \quad \forall k \geq 0. 
\end{equation}

\begin{theorem}
Let  the objective   function $f$ be convex and admit   a stochastic   $p \geq 1$ higher-order surrogate function $g$  as given in Definition \ref{def:1} and  $\bar L_p^h  = \max_{i=1:N} L_p^{h_i}$.    Moreover, let the sequence $(x_k)_{k\geq 0}$ generated by SHOM {with $\tau=1$} satisfy \eqref{eq:d} and $x_{k+1}$  be the global minimum of subproblem \eqref{eq:subproblem}.  Then,  $f(x_k)$ converges globally   in expectation to $f^*$ at a sublinear rate:
	$$
	\mathbb{E}[f(x_k)] - f^* \leq   \frac{\bar L_p^h  R^{p+1}}{p!}  \left(   \frac{p+1}{k} \right)^p. 
	$$
\end{theorem}

\begin{proof} 
Using Definition \ref{def:1} and $x_{k+1}$  the global minimum of subproblem \eqref{eq:subproblem},  we have:
	\begin{align*}
		f(x_{k+1}) \leq g(x_{k+1}; \hat{x}_k)  \leq g(x;\hat{x}_k )  \quad \forall x \in \mathcal{X}. 
	\end{align*} 
Furthermore, tacking  expectation and then the minimum over $\mathcal{X}$, we obtain:
\begin{align}
\mathbb{E}\left[f(x_{k+1})\right] &\leq \min_{x \in \mathcal{X}}\mathbb{E}\left[g(x;\hat{x}_k ) \right] = \min_{x \in \mathcal{X}}\mathbb{E}\left[  { \mathbb{E}_{k} } \left[g_i(x; x_k ) \right]\right] \nonumber\\
& \leq \min_{x \in \mathcal{X}}\mathbb{E}\left[   { \mathbb{E}_{k}} \left[ f_i(x) + \frac{L_p^{h_i}}{(p+1)!} \| x-x_k \|^{p+1} \right] \right] \nonumber\\
&  \leq  \min_{x \in \mathcal{X}}\mathbb{E} \left[ f(x) + \frac{\bar L_p^h}{(p+1)!} \| x-x_k \|^{p+1} \right], \label{eq:dltk}
\end{align}
where $\bar L_p^h  = \max_{i=1:N} L_p^{h_i}$.  Further,  our proof follows  a similar reasoning as in \cite{Nes:19} (Theorem 2).    For $k=1$, using the above inequality, we get: 
\begin{align*}
	\mathbb{E}\left[f(x_{1})\right] \leq \min_{x \in \mathcal{X}}\mathbb{E} \left[ f(x) + \frac{\bar L_p^h}{(p+1)!} \| x-x_0 \|^{p+1} \right] \overset{\eqref{eq:d}}{\leq} f^* + \frac{\bar L_p^h}{(p+1)!} R^{p+1},  
\end{align*}
and denoting  $\Delta_k = \mathbb{E}[f(x_k)] - f^*$,  we obtain:
\begin{align}
	\Delta_1 \leq \frac{\bar L_p^h}{(p+1)!} R^{p+1}.  \label{eq:dlt1}
\end{align}
Taking $x = \alpha x^* + (1-\alpha) x_k $ and using the convexity of $f$ in \eqref{eq:dltk}, we get for $k \geq 1$:
\begin{align*}
\mathbb{E}\left[f(x_{k+1})\right] \overset{\eqref{eq:d}}{\leq} \min_{\alpha \in [0,\,1]} \alpha f^* +  (1-\alpha) \mathbb{E}[f(x_k)] + \frac{\bar L_p^h}{(p+1)!} \alpha^{p+1} R^{p+1}. 
\end{align*}
The optimal $\alpha$ in the previous problem has the expression:
\begin{align}
	\label{eq:alfas}
	0< \alpha^* = \left(\frac{\Delta_k p!}{R^{p+1} \bar L_p^h} \right)^{\frac{1}{p}} \leq \left(\frac{\Delta_1 p!}{R^{p+1} \bar L_p^h} \right)^{\frac{1}{p}}\overset{\eqref{eq:dlt1}}{\leq}   \left( \frac{1}{p+1}  \right)^{\frac{1}{p}}< 1 \quad \forall  k\geq 1. 
\end{align}
Thus, we conclude that:
\begin{align*}
&\mathbb{E}\left[f(x_{k+1})\right] \leq \mathbb{E}\left[f(x_{k})\right] -  \frac{p}{p+1} \alpha^* \Delta_k \Leftrightarrow \Delta_{k+1} \leq \left( 1- \frac{p}{p+1} \alpha^* \right)\Delta_k \\
& \Leftrightarrow \Delta_k - \Delta_{k+1} \geq  \frac{p}{p+1} \alpha^* \Delta_k \overset{\eqref{eq:alfas}}{=} \frac{p}{p+1} \left(\frac{p!}{R^{p+1} \bar L_p^h}\right)^{\frac{1}{p}} \Delta_k^{\frac{p+1}{p}}. 
\end{align*}
Using further the same arguments as in Theorem 2 in \cite{Nes:19}  for the previous recurrence we obtain the statement. 
\end{proof}

\noindent  Note that in the previous theorem we do not require each individual function $f_i$, with $i=1:N$, to be convex, only  the finite sum function $f$ needs to be convex.  In the sequel we  prove, that SHOM can achieve locally faster rates  when  the finite-sum objective function $f$ is uniformly convex. More precisely,  we get   \textit{local (super)linear} convergence rates for SHOM in several optimality  criteria:  function values and distance of the iterates to the optimal point. For this we first need some auxiliary results.  Let $x^*$ be the unique solution of \eqref{fw:eq1} and assume, for simplicity, that  $L_p^{h_j} = L_p^{h}$ for all $j=1:N$. 	

\begin{lemma} \label{tm:1}
	Let $f$ be uniformly convex of degree $q \geq 2$ and constant $\sigma_q$ and  admit a stochastic   $p \geq 1$ higher-order surrogate function $g$ over  $ \mathcal{X}$  as given in Definition \ref{def:1}.  Then, the sequence $\left(x_{k}\right)_{k \geq 0}$ generated by SHOM satisfies:
	\begin{equation} \label{eq:c14}
		\frac{\sigma_{q}}{q}\|x_{k+1} - x^* \|^{q} \leq f(x_{k+1}) - f(x^*) \leq  \frac{L_p^{h}}{N(p+1)!} \sum_{j=1}^{N} \|x_k^j - x^* \|^{p+1} .
	\end{equation}
\end{lemma}

\begin{proof}
	Since  $h_j(y;x_k^j) = g_j(y;x_k^j) - f_j(y)$ has the $p$ derivative smooth with Lipschitz constant $L_p^{h}$, then from \eqref{eq:TayAppBound} we have:
	\begin{align*}
		|h_j(y;x_k^j) - T_p^{h_j}(y;x_k^j)| \leq \frac{L_p^{h}}{(p+1)!} \|y-x_k^j\|^{p+1}\quad \forall y.
	\end{align*}
	Note that from property 3 of Definition \ref{def:1}, we  have: 
		$$T_p^{h_j}(y;x_k^j) = h_j(x_k^j;x_k^j) + \sum_{i=1}^p \frac{1}{i!} D^i h_j(x_k^j;x_k^j)[y-x_k^j]^i= 0, \quad \quad \forall \,\, j=1:N.$$ Therefore, summing over  $j$ in the previous inequality, we obtain:
	\begin{align*}
		\frac{1}{N}\sum_{j=1}^{N} |h_j(y;x_k^j)| \leq \frac{1}{N} \frac{L_p^{h}}{(p+1)!} \sum_{j=1}^{N} \|y-x_k^j\|^{p+1} .
	\end{align*}
	Further, since  the surrogate function  $h_j( \cdot; x_k^j) \geq 0$, we get:
	\begin{align}\label{eq:s1}
		h(y;\hat{x}_k) \leq \frac{L_p^{h}}{N(p+1)!} \sum_{j=1}^{N} \|y-x_k^j\|^{p+1}.
	\end{align}
	On the other hand, from the global optimality of  $x_{k+1}$, we have:
	\begin{align*} 
		f(x_{k+1}) \leq g(x_{k+1}; \hat{x}_k) \leq g(y; \hat{x}_k) = f(y) + h(y;\hat{x}_k) \quad \forall  y \in \mathcal{X}.
	\end{align*}
	Using \eqref{eq:s1} in the above inequality, we obtain:
	\begin{align*}
		f(x_{k+1})\leq f(y) + \frac{L_p^{h}}{N(p+1)!} \sum_{j=1}^{N} \|y-x_k^j\|^{p+1} \quad \forall  y \in \mathcal{X}.
	\end{align*}
	Choosing $y = x^*$ we obtain the right side of  \eqref{eq:c14}.
	On the other hand, from the uniform convexity relation \eqref{eq:unifConv} for the function $f$  at $x = x^*$, we have:
	\begin{equation*}
		f(x_{k+1}) \geq f(x^*)+\frac{\sigma_{q}}{q}\|x^*-x_{k+1}\|^{q},
	\end{equation*}
	which yields the left side of relation \eqref{eq:c14}. This proves our statements.  
\end{proof}


\subsection{Local (super)linear convergence in function values}
In this section we prove local (super)linear convergence in function values for SHOM, provided that the objective function is uniformly convex. Note that in Lemma \ref{th:nonconv-gen} we have proved already that the objective function decreases along the iterates of SHOM.  From Lemma  \ref{tm:1} we obtain the following relation:
\begin{lemma} \label{lm:c2}
	Let $f$ be uniformly convex of degree $q \geq 2$ with constant $\sigma_q$ and  admit a stochastic   $p \geq 1$ higher-order surrogate function $g$ over  $ \mathcal{X}$  as given in Definition \ref{def:1}.  Consider the following Lyapunov function {along the sequence $\hat{x}_k$ generated by SHOM}\footnote{ {By Lyapunov function along the sequence $\hat{x}_k$ generated by SHOM we mean that we evaluate the Lyapunov function $	\Phi(\hat{x}) = \frac{1}{N} \sum_{j=1}^{N}(f(x^j) - f(x^*))^{\frac{p+1}{q}}$ at $\hat{x}_k$. }}:
	$$
	\Phi_k = \frac{1}{N} \sum_{j=1}^{N}(f(x_k^j) - f(x^*))^{\frac{p+1}{q}}.
	$$
	Then, we have:
	\begin{equation} \label{eq:c15}
		f(x_{k+1}) - f(x^*)\leq  \frac{L_p^{h}}{(p+1)!} \left(\frac{q}{\sigma_{q}}\right)^{\frac{p+1}{q}} \Phi_{k}.
	\end{equation}
\end{lemma}

\begin{proof}
	From the left hand side of relation \eqref{eq:c14},  we have	for all $j =1:N$ that:
	\begin{align}
		\left[\frac{q}{\sigma_{q}}(f(x_k^j) -f(x^*)) \right]^{\frac{1}{q}} \geq \|x^* - x_k^j \|.
	\end{align}
	Returning with this relation into the right side of \eqref{eq:c14}, we obtain: 
	\begin{equation}\label{eq:c9}
		f(x_{k+1}) - f(x^*)\leq  \frac{L_p^{h}}{N(p+1)!} \left(\frac{q}{\sigma_{q}}\right)^{\frac{p+1}{q}} \sum_{j=1}^{N} (f(x_k^j) -f(x^*))^{\frac{p+1}{q}}. 
	\end{equation}
	Thus, proving the  relation  \eqref{eq:c15}.  
\end{proof}


\noindent   Note that a similar Lyapunov function for $p=2$ and $q=2$ was considered in   \cite{KovRic:19}.

\begin{lemma}
	Let $f$ be uniformly convex of degree $q \geq 2$ with constant $\sigma_q$ and  admit a stochastic   $p \geq 1$ higher-order surrogate function $g$ over  $ \mathcal{X}$  as given in Definition \ref{def:1}.  Furthermore, let $\alpha \in \left(0, \, 1\right)$ and $C = \frac{L_p^{h}}{(p+1)!} \left(\frac{q}{\sigma_{q}}\right)^{\frac{p+1}{q}} $.  If  $q < p+1 $, then the following implication holds with probability 1: 
	\begin{equation}\label{eq:c16}
		f(x_{0}) - f(x^*) \leq \frac{\alpha^{\frac{q^2}{(p+1)^2-q(p+1)}}}{C^{\frac{q}{p+1 - q}}} \quad \Rightarrow \quad  
		\Phi_{k} \leq \frac{\alpha^{\frac{q}{p+1-q}}}{C^{\frac{p+1}{p+1 - q}}} \quad \forall k\geq 0. 
	\end{equation}
\end{lemma}

\begin{proof}
	The statement  follows as a corollary of a stronger result, that is  for any $j=1:N$ and $k \geq 0$ we have with probability $1$ that: 
	$$
	(f(x_k^j) - f(x^*))^{\frac{p+1}{q}} \leq \frac{\alpha^{\frac{q}{p+1-q}}}{C^{\frac{p+1}{p+1 - q}}} \quad \forall j=1:N.
	$$
	Let us show this by induction in $k$.  For $k = 0$ it is satisfied by our assumption. Let us assume that it  holds for $k$.  If $j \notin S_{k+1}$, then  $x_{k+1}^j = x_{k}^j$ and from the induction step $k$ we also get that it is valid for $k+1$.  If  $j \in S_{k+1}$, we have: 
	\begin{align*}
		&f(x_{k+1}^j) - f(x^*) = f(x_{k+1}) -f(x^*) \overset{\eqref{eq:c9}}{\leq} \frac{C}{N} \sum_{j=1}^{N} (f(x_{k}^j) - f(x^*))^{\frac{p+1}{q}}\\
		& \leq \frac{C}{N}\sum_{j=1}^{N} \frac{\alpha^{\frac{q}{p+1-q}}}{C^{\frac{p+1}{p+1 - q}}} =  \left(\frac{\alpha}{C} \right)^{\frac{q}{p+1-q}} \leq  \frac{\alpha^{\frac{q^2}{(p+1)(p+1-q)}}}{C^{\frac{q}{p+1-q}}},
	\end{align*}
	as  $\alpha \in \left(0, \, 1\right)$ and  $q < p+1 $.   This shows that:  
	$$\left(f(x_{k+1}^j) - f(x^*)\right)^{\frac{p+1}{q}} \leq \frac{\alpha^{\frac{q}{p+1-q}}}{C^{\frac{p+1}{p+1 - q}}} \quad \forall j=1:N. $$ This concludes our induction and the statement of the lemma.     
\end{proof}


\begin{theorem}
	\label{th:fv}
	Let $f$ be uniformly convex of degree $q \geq 2$ with constant $\sigma_q$ and  admit a stochastic   $p \geq 1$ higher-order surrogate function $g$ over  $ \mathcal{X}$  as given in Definition \ref{def:1}.  Then, the  Lyapunov function {along the sequence $\hat{x}_k$ generated by SHOM}, $	\Phi_k = \frac{1}{N} \sum_{j=1}^{N}(f(x_k^j) - f(x^*))^{\frac{p+1}{q}}$, satisfies the recursion:
	\[
	\mathbb{E}_{k}\left[\Phi_{k+1}\right] \leq\left[1-\frac{\tau}{N}+\frac{\tau}{N} \left( \frac{ q^{\frac{p+1}{q}}L_p^{h}}{\sigma_{q}^{\frac{p+1}{q}}(p+1)!} \right)^{\frac{p+1}{q}} \Phi_{k}^{\frac{p+1-q}{q}} \right] \Phi_{k}.
	\]
	Furthermore, let $\alpha \in \left(0, \, 1\right)$ and $C = \frac{L_p^{h}}{(p+1)!} \left(\frac{q}{\sigma_{q}}\right)^{\frac{p+1}{q}} $.  If  $q < p+1 $ and   $f(x_{0}) - f(x^*) \leq \frac{\alpha^{\frac{q^2}{(p+1)^2-q(p+1)}}}{C^{\frac{q}{p+1 - q}}}$, then we have the following linear rate:
	\[
	\mathbb{E}_{k}\left[\Phi_{k+1}\right] \leq\left(1-\frac{(1-\alpha)\tau}{  N}\right) \Phi_{k}.
	\]
\end{theorem}

\begin{proof}
	Let us note that for all $j=1:N$, the points $x_{k+1}^j$  are drawn  from the following conditional probability distribution:
	\begin{align}
		\label{eq:prob}
		\mathbb{P}\left(x_{k+1}^{j}=x_{k+1} \mid \mathcal{F}_{k}\right)=\delta \text { and } \mathbb{P}\left(x_{k+1}^{j}=x_{k}^{j} \mid \mathcal{F}_{k}\right)=1-\delta,
	\end{align}	
	where $\delta:= \tau / N$ and $x_{0}^{j} = x_{0}$ for all $j=1:N$. Thus, we have for all $k$: 
	\begin{align*}
		\mathbb{E}_k\left[\Phi_{k+1}\right] & =\delta \mathbb{E}_k\left[\left( f(x_{k+1}) - f(x^*)\right)^{\frac{p+1}{q}}\right]+(1-\delta) \Phi_{k} \overset{\eqref{eq:c15}}{\leq}\delta \mathbb{E}_k \left[\left(C \Phi_{k}\right)^{\frac{p+1}{q}}\right] + (1-\delta) \Phi_{k} \\
		& = \delta\left(C \Phi_{k}\right)^{\frac{p+1}{q}}+ (1-\delta) \Phi_{k} \!= \!\left(\!1-\delta + \delta C^{\frac{p+1}{q}} \Phi_k^{\frac{p+1-q}{q}} \right) \! \Phi_{k}  \overset{\eqref{eq:c16}}{\leq}  \!\!
		\left( 1-(1-\alpha)\delta \right) \Phi_{k}.
	\end{align*}
	This proves our statements. 
\end{proof}

\begin{remark}  
	\label{rem:comp}
Although for stochastic first-order variance reduction methods, such as SAGA 	\cite{DefBac:14}, SVRG \cite{JohZha:13} or SARAH \cite{NguSch:17},  one can also obtain linear convergence rates, they depend on the condition number of the problem and consequently they converge slow on ill-conditioned problems.   	On the other hand, the convergence rates from Theorem \ref{th:fv}   are independent of the conditioning, thus, SHOM provably adapts to the curvature of the objective function. For example, for the particular case $p = q =2$ and $\alpha = \frac{1}{2}$ we get  a very fast linear rate depending on the minibatch size $\tau$: 
	\begin{equation}
		\label{eq:ind}
		\mathbb{E}_k\left[\Phi_{k+1}\right] \leq \left(1 - \frac{\tau}{2N}\right) \Phi_{k}.
	\end{equation}
A similar   convergence rate as in 	\eqref{eq:ind} was  derived in \cite{KovRic:19}  for an algorithm similar to SHOM, but for the particular surrogate function of Example \ref{expl:5} with $p =2$  (i.e., second-order Taylor expansion with cubic regularization) and strongly convex objective (i.e., $q=2$).  Our results cover more general problems (including composite objective, see Examples  \ref{expl:6},  \ref{expl:7},  \ref{expl:8}) and uniform convex cost. 

\medskip 

\noindent Further,  from Lemma  \ref{lm:c2} and Theorem \ref{th:fv}  
we get that  $\mathbb{E} [  f({x}_k) ] - f(x^*) $ converges linearly to $0$, locally.  Moreover,  when $\tau = N = 1$  SHOM becomes  deterministic and for such algorithm   we get from the first statement of Theorem \ref{th:fv}  local  superlinear convergence in function values, i.e. for $q < p+1$ we get: 
	$$
	f(x_{k+1}) - f(x^*) \leq  C(f(x_k) - f(x^*))^{\frac{p+1}{q}}. 
	$$ 
	A similar convergence result has been obtained  in  \cite{DoiNes:2019} for a deterministic algorithm derived from the particular surrogate function of Example  \ref{expl:6}  with $N=1$ (Taylor expansion with regularization).  \textit{However, our convergence analysis of SHOM  is derived for general surrogates. Note that there  is a major difference between the Taylor expansion and the model approximation based on  general surrogates (majorization-minimization framework). Taylor expansion is unique. The idea of the majorization-minimization approach is the opposite. A given objective function may admit  many upper bound models and every model leads to a different optimization algorithm.}  \qed
\end{remark}



\subsection{Local (super)linear convergence in the iterates}
\noindent In this section we prove local linear convergence in the iterates of SHOM, provided that the objective function is uniformly convex.   From Lemma  \ref{tm:1} we obtain immediately the following relation:
\begin{lemma} \label{lm:c3}
	Let $f$ be uniformly convex of degree $q \geq 2$ with constant $\sigma_q$ and  admit a stochastic   $p \geq 1$ higher-order surrogate function $g$ over  $ \mathcal{X}$  as in Definition \ref{def:1}.  If we consider the following Lyapunov function {along the sequence $\hat{x}_k$ generated by SHOM}:
	$$
	\mathcal{W}_{k} = \frac{1}{N} \sum_{j=1}^{N} \| x^j_{k} - x^* \|^{p+1}.
	$$
	then, we have:
	\begin{equation} \label{eq:c10}
		\|x_{k+1} - x^*\|^q \leq \frac{L_p^{h}q}{\sigma_{q}(p+1)!} \mathcal{W}_{k}.
	\end{equation}
\end{lemma}


\begin{lemma}
	Let $f$ be uniformly convex of degree $q \geq 2$ with constant $\sigma_q$ and  admit a stochastic   $p \geq 1$ higher-order surrogate function $g$ over  $ \mathcal{X}$  as given in Definition \ref{def:1}. Furthermore, let $\alpha \in (0, \, 1)$ and  $\bar{C} = \frac{L_p^{h}q}{\sigma_{q}(p+1)!}$.  If $q< p+1$, then the following implication holds with probability $1$: 
	\begin{equation} \label{eq:c13}
		\|x_0 -x^* \| \leq \frac{\alpha^{\frac{q}{(p+1)^2-q(p+1)}}}{\bar{C}^{\frac{1}{p+1-q}}} \quad \Rightarrow \quad 	\mathcal{W}_k \leq \frac{\alpha^{\frac{q}{p+1-q}}}{\bar{C}^{\frac{p+1}{p+1-q}}} \quad \forall k \geq 0. 
	\end{equation}
\end{lemma}

\begin{proof}
	The statement  follows as a corollary of a stronger statement, that is for any $j=1:N$ and $k \geq 0$ we have with probability $1$ that: 
	\begin{equation} \label{eq:c11}
		\left\|x^{j}_{k}-x^{*}\right\|^{p+1} \leq \frac{\alpha^{\frac{q}{p+1-q}}}{\bar{C}^{\frac{p+1}{p+1-q}}}.
	\end{equation}
	We will  prove this relation by induction over $k$.  Clearly,  \eqref{eq:c11} holds for $k=0$ from our assumption. Assume it holds for $k$ and let us show that it holds for $k+1$. If $j \notin S_{k+1}$, then  $x_{k+1}^j = x_{k}^j$  and  \eqref{eq:c11} holds from our  induction assumption. On the other hand, for $j \in S_{k+1}$, we have:
	\begin{align*}
		\|x_{k+1}^{j} - x^* \|^{p+1} &= \|x_{k+1} - x^* \|^{p+1} \overset{\eqref{eq:c10}}{\leq} \left(\frac{\bar{C}}{N} \sum_{j=1}^{N} \|x^* -x_k^j \|^{p+1}\right)^{\frac{p+1}{q}} \\
		& \overset{\eqref{eq:c11}}{\leq}\left(\frac{\bar{C}}{N}  \sum_{j=1}^{N}\frac{\alpha^{\frac{q}{p+1-q}}}{\bar{C}^{\frac{p+1}{p+1-q}}}\right)^{\frac{p+1}{q}} \leq \frac{\alpha^{\frac{p+1}{p+1-q}}}{\bar{C}^{\frac{p+1}{p+1-q}}} \leq \frac{\alpha^{\frac{q}{p+1-q}}}{\bar{C}^{\frac{p+1}{p+1-q}}},
	\end{align*}
	as $\alpha \in (0, \, 1)$ and $p+1 > q$.  Thus, relation \eqref{eq:c13} holds.    
\end{proof}


\begin{theorem}
	\label{th:iter}  
	Let $f$ be uniformly convex of degree $q \geq 2$ with constant $\sigma_q$ and  admit a stochastic   $p \geq 1$ higher-order surrogate function $g$ over  $ \mathcal{X}$  as given in Definition \ref{def:1}.  Then, the  Lyapunov function {along the sequence $\hat{x}_k$ generated by SHOM}, $	\mathcal{W}_{k} = \frac{1}{N} \sum_{j=1}^{N} \|x^* - x^j_{k} \|^{p+1}$,  satisfies the recursion: 
	$$\mathbb{E}_k\left[\mathcal{W}_{k+1}\right] \leq \left[1- \frac{\tau}{N} + \frac{\tau}{N} \left(\frac{L_p^{h}q}{\sigma_{q}(p+1)!}\right)^{\frac{p+1}{q}} \mathcal{W}_k^{\frac{p+1-q}{q}}\right] \mathcal{W}_k.
	$$
	Furthermore,   let  $\alpha \in \left(0,\, 1\right)$ and $\bar{C} =  \frac{L_p^{h}q}{\sigma_{q}(p+1)!}$.  If $q< p+1$  and $\|x_0 -x^* \| \leq \frac{\alpha^{\frac{q}{(p+1)^2-q(p+1)}}}{\bar{C}^{\frac{1}{p+1-q}}}$, then we have the following linear rate:
	$$ 
	\mathbb{E}_k\left[\mathcal{W}_{k+1}\right] \leq \left( 1-\frac{(1-\alpha) \tau}{N}\right) \mathcal{W}_k.
	$$
\end{theorem}

\begin{proof}
	From  \eqref{eq:prob}  we have for any $k$ that:
	\begin{align*}
		\mathbb{E}_k\left[\mathcal{W}_{k+1}\right] & =\delta \mathbb{E}_k\left[ \| x_{k+1} - x^* \|^{p+1}\right]+(1-\delta) \mathcal{W}_{k}\\
		&=\delta \mathbb{E}_k\left[\left(\| x_{k+1} - x^* \|^{q}\right)^{\frac{p+1}{q}} \right]+(1-\delta) \mathcal{W}_{k}\\
		& \overset{\eqref{eq:c10}}{\leq} \delta \mathbb{E}_k\left[\left(\bar{C} \mathcal{W}_{k}\right)^{\frac{p+1}{q}} \right]+(1-\delta) \mathcal{W}_{k} \\
		& = \delta \left( \bar{C} \mathcal{W}_{k}\right)^{\frac{p+1}{q}}+(1-\delta) \mathcal{W}_{k} = \left(1-\delta + \delta \bar{C}^{\frac{p+1}{q}} \mathcal{W}^{\frac{p+1-q}{q}}\right) \mathcal{W}_k  \\
		& \overset{\eqref{eq:c13}}{\leq} \left( 1-\frac{(1-\alpha)\tau}{N}\right) \mathcal{W}_k.
	\end{align*}
	This proves our statements. 
\end{proof}

\noindent Note that for any particular choice of $\alpha$ and $\tau$ we obtain from Theorem \ref{th:iter}  linear rate independent of the conditioning, thus, the method provably adapts to the curvature of the  function $f$. For example, $p = q =2$ and $\alpha =\frac{1}{4}$ 
yields:
$$ 
\mathbb{E}_k\left[\mathcal{W}_{k+1}\right] \leq \left( 1-\frac{3\tau}{4N}\right) \mathcal{W}_k.
$$  Moreover, from Lemma \ref{lm:c3} and Theorem \ref{th:iter} we get that  $\mathbb{E} [ \| {x}_{k} - x^* \|^{q} ] $ converges linearly to $0$, locally.   Furthermore,  when $\tau = N=1$  SHOM becomes deterministic and for such algorithm   we obtain from the first statement of Theorem \ref{th:iter}  local  superlinear convergence in the iterates, i.e. for  $q<p+1$  we get: 
$$
\|x_{k+1} - x^* \| \leq   \bar{C}^{\frac{1}{q}}    \|x_k - x^* \|^{\frac{p+1}{q}}.  
$$


\section{Numerical tests}
In this section we consider two applications, a classification problem based on the logistic  loss, which can be formulated as a convex finite-sum problem,  and  the independent component analysis problem,  formulated as a nonconvex finite-sum problem. {In all our numerical experiments we use the metric $B = I_n$ (the identity matrix).}

\subsection{Implementation details for SHOM }  
\label{sec:impl}
In this section we show that for structured finite-sum problems, such as composite problems with structure (see Example \ref{expl:6} and also the applications below),  SHOM subproblem can be solved  efficiently. More specifically, for given vectors $a_i$'s,  let us consider each term in the finite-sum of the form $f_i(x) =  \phi_i(a_i^Tx)  + \psi(x)$, where  tipically $ \phi_i$  are sufficiently smooth  (possibly nonconvex) functions (see  Example \ref{expl:2} and the applications below) and $\psi$ is a simple function such as  a regularizer or  indicator function of some (possibly nonconvex) set (see the problems \eqref{eq:logistic} or \eqref{eq:ica_pb} below). In this case  we can easily compute higher-order directional derivatives along a direction  $h$, e.g.,    
$D \phi_i(a_i^Tx)[h] = \phi_i^{\prime}(a_i^Tx) \cdot a_i^T h$, $D^2 \phi_i(a_i^Tx)[h]  = \phi_i^{\prime \prime}(a_i^Tx) a_i^Th \cdot a_i$ and $D^3 \phi_i(a_i^Tx)[h]^2  =  \phi_i^{\prime \prime \prime}(a_i^Tx)(a_i^Th)^2 \cdot a_i$.  Hence, in these settings, with  $p \in \{1, 2, 3\}$ and $M_p^j = M_p$ for all $j =1:N$, the surrogate of  each $f_j$ is:
$$
g_j(y;x_k^j) = T_p^{\phi_j}(y;x_k^j) + \frac{M_p}{(p+1)!} \|y-x_k^j \|^{p+1} + \psi(y). 
$$
 Note that we can  update very efficiently the model approximation $g$ in  SHOM algorithm using \eqref{eq:model-aprox}. 	
Then, the iteration of  SHOM requires solving a subproblem of the following form:
\begin{equation} 
		\label{sim:eq10}
	x_{k+1} = \!\argmin_{y} g(y;\hat{x}_k) = \frac{1}{N} \sum_{j=1}^{N}  \left( T_p^{\phi_j}(y;x_k^j) + \frac{M_p}{(p+1)!} \|y-x_k^j \|^{p+1}\right) +\psi(y). 
\end{equation}
Moreover,   subproblem \eqref{sim:eq10} can be solved tipically in closed form  for $p=1$. For  $p \in \{2, 3\}$ one can use  the powerful tools of convex optimization when the subproblem is (uniformly) convex (see e.g.  \cite{Nes:19,Nes:Inexat19,DefBac:14}),  IPOPT \cite{ipopt} or can use e.g., ADMM, which is known to converge also fast \cite{BoyPar:11}, when the subproblem is nonconvex.  Below, we briefly sketch  the main steps of ADMM.  More precisely, after a change of variables $z=y-x_k$,  subproblem \eqref{sim:eq10} can be written equivalently as:
\begin{equation} 
	\label{sim:eq1}
	\min_{z}  \sum_{\ell =1}^p \frac{1}{\ell !} \Delta_\ell [z]^\ell + \frac{\tau M_p}{N(p+1)!} \|z \|^{p+1} + \bar \psi(z) + \frac{1}{N} \sum_{j \not \in S_k}  \frac{M_p}{(p+1)!} \|z+x_k-x_k^j \|^{p+1}, 
\end{equation}
where $\Delta_\ell[z]^{\ell}$ are computed from  higher-order directional derivatives of $\phi_j$'s along directions $x_k^j$  and/or $z$ as explained above. Hence, we need to keep track of only scalar values, no need to memorize matrices or tensors.  Further, we can reformulate \eqref{sim:eq1}  as the following problem:
\begin{align} 
	\label{sim:eq2}
	& \min_{z, z^j}  \sum_{\ell =1}^p \frac{1}{\ell !} \Delta_\ell [z]^\ell +  \frac{\tau M_p}{N(p+1)!} \|z \|^{p+1}  + \bar \psi(z) +  \frac{1}{N} \sum_{j \not \in S_k}   \frac{M_p}{(p+1)!} \|z^j \|^{p+1}  \\
	&  \text{s.t.} \quad z^j=z + x_k - x_k^j \quad \forall j \not \in S_k.  \nonumber 
\end{align}
At each iteration of ADMM one needs to solve a subproblem in $z$ of the form:
\[  	\min_{z}  \sum_{\ell =1}^p   \frac{1}{\ell !}  \bar \Delta_\ell [z]^\ell +  \frac{\tau M_p}{N(p+1)!} \|z \|^{p+1} + \bar \psi(z),  \]
for which there are efficient algorithms when $p \in \{2,3\}$  (see e.g., \cite{NesBor:06,Nes:Inexat19}); and  $N -\tau$ subproblems in $z^j$ of the form (here $\Lambda_j \in \mathbb{R}^n$ and $\rho>0$):
\[  	\min_{z^j}   \langle \Lambda_j,  z^j \rangle + \frac{\rho}{2}\|z^j\|^2 + \frac{M_p}{(p+1)!} \|z^j \|^{p+1}, \]
which can be solved in parallel and in closed form. More specifically,   the optimality condition yields $\Lambda_j + \rho z^j + \frac{M_p}{(p)!} \|z^j \|^{p-1} z^j=0 $. One can notice that if $\|z^j \|^{p-1}$ is known, then one can easily compute the optimal $z^j$ from the previous equality. Hence, taking the norm in the optimality condition we get an equation in $\|z^j\|$ of the form $\|\Lambda_j\| =  \rho \|z^j\| + \frac{M_p}{(p)!} \|z^j \|^{p}$, whose roots can be found explicitly for $p \in \{2, 3\}$. Finally, the Lagrange multipliers are updated appropriately, see \cite{BoyPar:11}.   { We can stop the algorithm when the feasibility constraints in \eqref{sim:eq2} are below some desired accuracy or after  a maximum number of  ADMM steps.}

  


\subsection{Classification problem}
The problem of  fitting a generalized linear model with $l_2$ 
regularization can be formulated as \cite{GooBen:16}: 
\begin{equation}
	\label{eq:logistic}
	\min_{x} f(x) =  \frac{1}{N} \sum_{i=1}^{N} \underbrace{\phi(a_i^Tx)}_{=f_i(x)} + \frac{\lambda}{2}\|x\|^2,
\end{equation}
where $a_{1}, \ldots, a_{N} \in \mathbb{R}^{n} $ are given vectors ,  $\lambda > 0$ is a  fixed scalar,  and $\phi: \mathbb{R} \to \mathbb{R}$ are convex functions, at least three times differentiable and sufficiently smooth. One such example is the logistic function $\phi(t) = \log(1 + e^{t})$, which appears often in machine learning applications, e.g.,  classification or regression tasks \cite{GooBen:16}.   We also consider this function for numerical experiments and  two  datasets  taken from  LIBSVM  \cite{Libsvm}:  madelon ($N = 2.000$ data with $n=500$ features) and a8a (the first $N = 5.000$  data with $n = 122$ features).  The regularization parameter  is $\lambda = 10^{-3}$ and the other parameters are:  {$L_1^{f_i} = \frac{1}{4} \| a_i\|^2\, (M_1 \equiv M_1^i = \max_{i=1:N} L_1^{f_i})$, $L_2^{f_i} = \frac{1}{3} \| a_i\|^3\, (M_2 \equiv M_2^i =  \max_{i=1:N} 2L_2^{f_i})$ and $L_3^{f_i} = \frac{2}{3} \sum_{i=1}^{N} \| a_i\|^4 \,(M_3 \equiv M_3^i = \max_{i=1:N} 3L_3^{f_i})$, for all $i=1:N$, respectively.}   

\begin{figure}[!h]
	\includegraphics[height= 0.2\textheight, width=0.3\textheight]{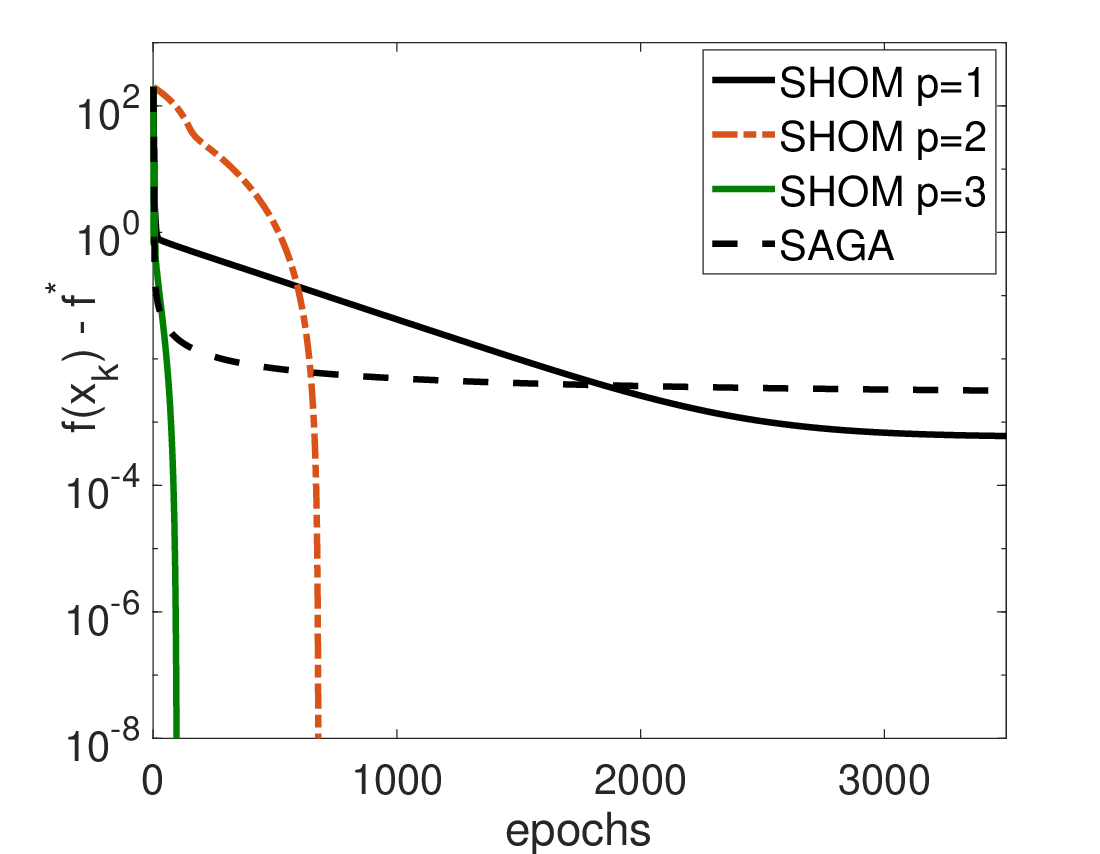}
	\hspace{-0.3cm}
	\includegraphics[height= 0.2\textheight, width=0.31\textheight]{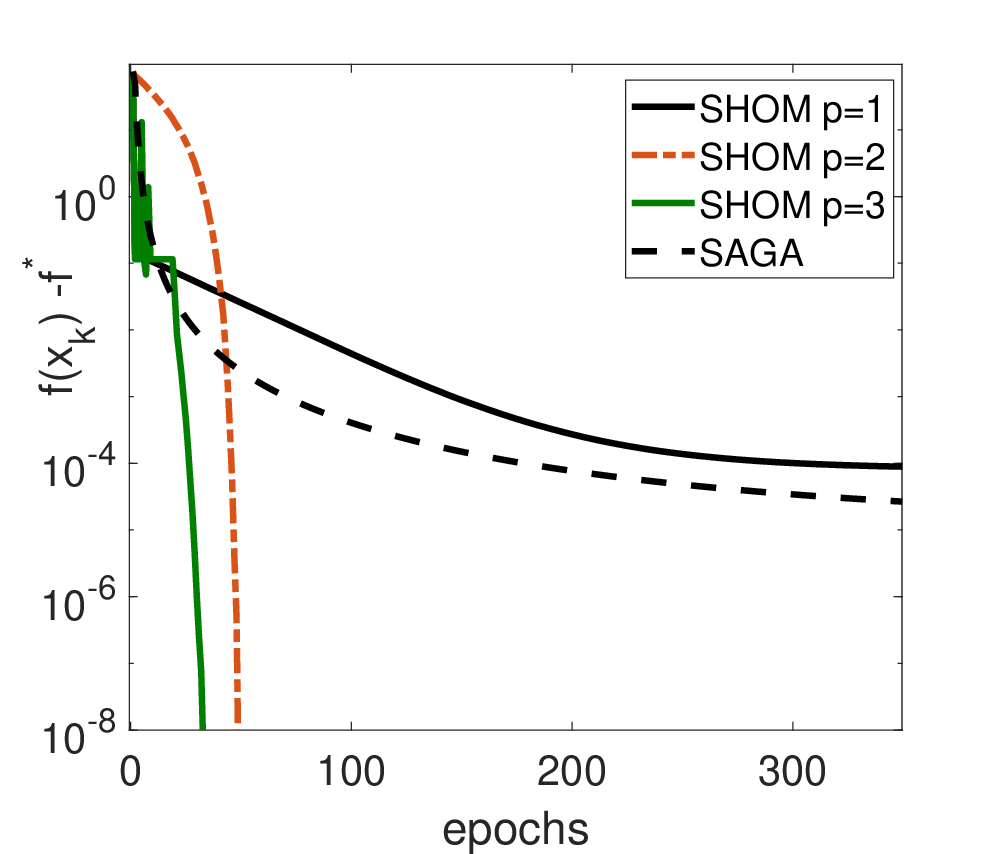}
	\caption{Behavior of SHOM ($p=1,2,3$) and SAGA: left - a8a, right - madelon.}
\end{figure}

\noindent We first analyse the behavior of  algorithm SHOM for different values of $p \in \{1, 2, 3\}$.  We also compare the   SHOM variants with SAGA, a first-order variance reduction algorithm 	\cite{DefBac:14}.  In this experiment we consider the minibatch size  {$\tau=300$ for SHOM and SAGA}. The results are given in Figure 1, which shows that for a relatively small minibatch size the  increasing of the approximation order $p$ in SHOM algorithm  has beneficial effects in terms of convergence rates (less number of epochs, i.e. number of passing through data) and in terms of achieving very high accuracies. Moreover, SHOM algorithm with $p>1$  is performing better than SAGA or SHOM with $p=1$, {i.e., much less number of epochs and less CPU time (see Table \ref{table:1}).}  

\begin{table}
\centering	
\caption{{CPU time in seconds for algorithms SHOM $(p = 1,\, 2,\, 3)$ and SAGA on dataset $a8a$.}}
	\begin{tabular}{|c|c|c|c|c|}
	\hline
	Alg. & SHOM $p=3$ & SHOM $p=2$ & SHOM $p=1$ & SAGA \\
	\hline
	CPU time & 268 & 283 & 305 & 296 \\
	\hline
	\end{tabular}
\label{table:1}
\end{table}

\begin{figure}[!h]
	\centering
	\includegraphics[height= 0.18\textheight,width=0.31\textheight]{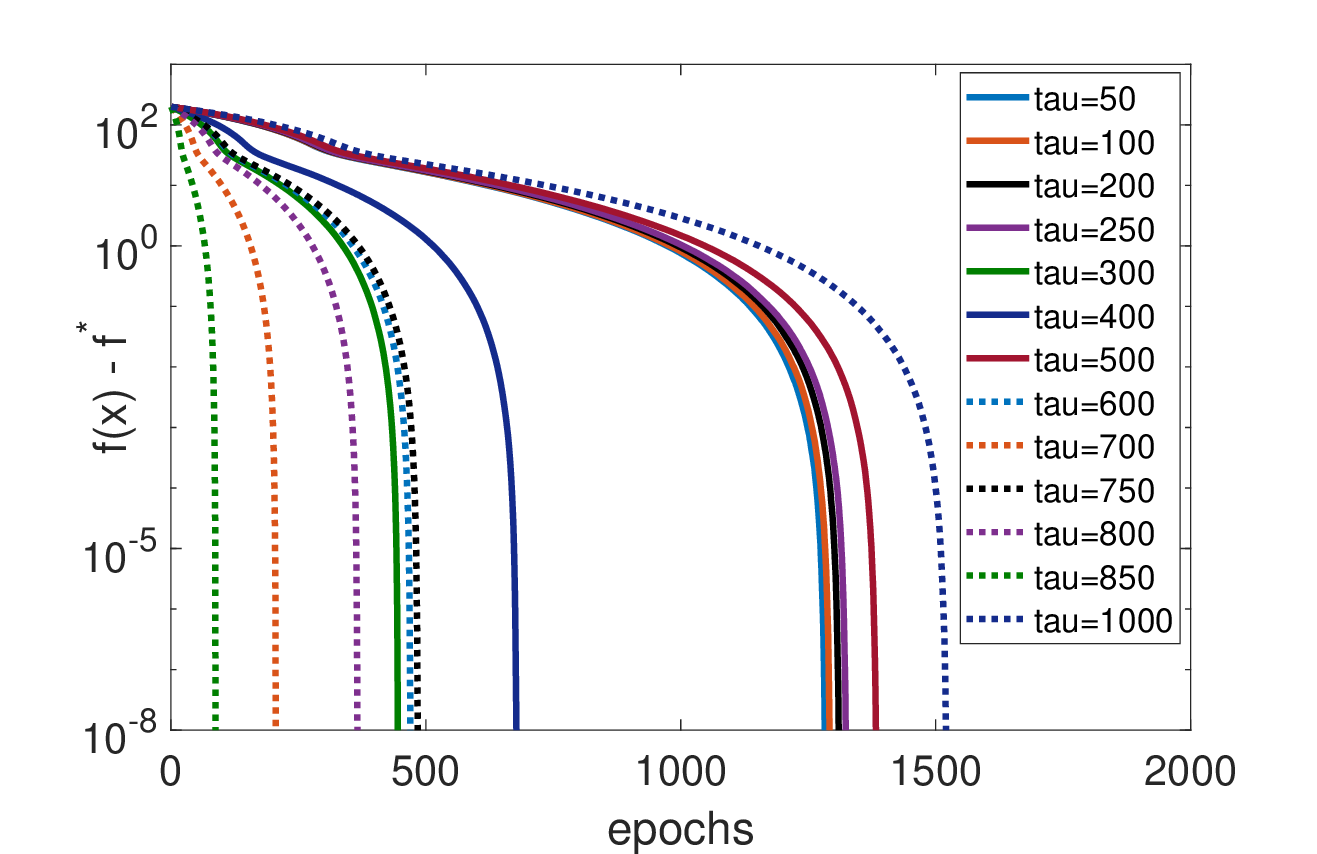} 
	\hspace{-0.5cm}
	\includegraphics[height= 0.18\textheight, width=0.31\textheight]{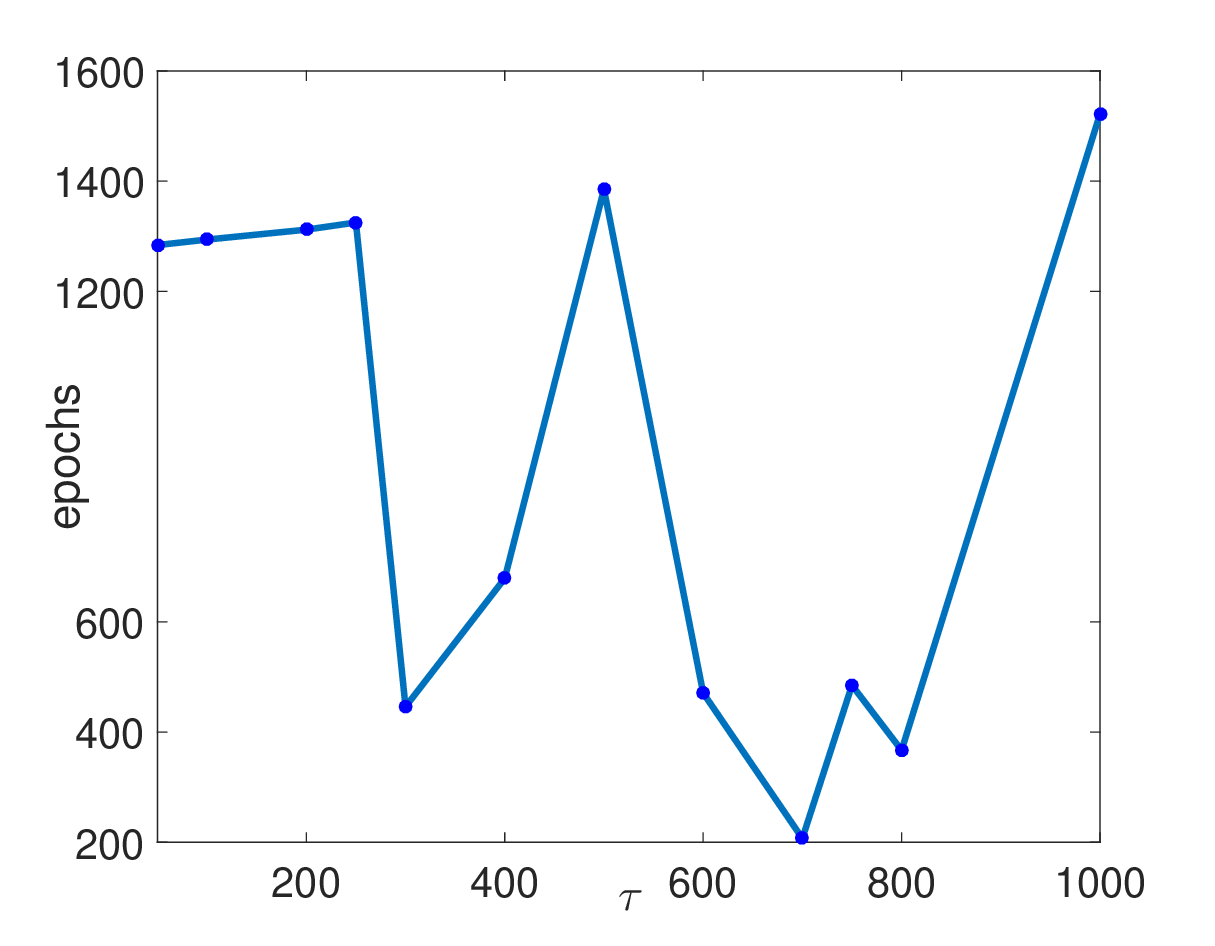}

\vspace{-0.1cm}
	
	\includegraphics[height= 0.18\textheight, width=0.31\textheight]{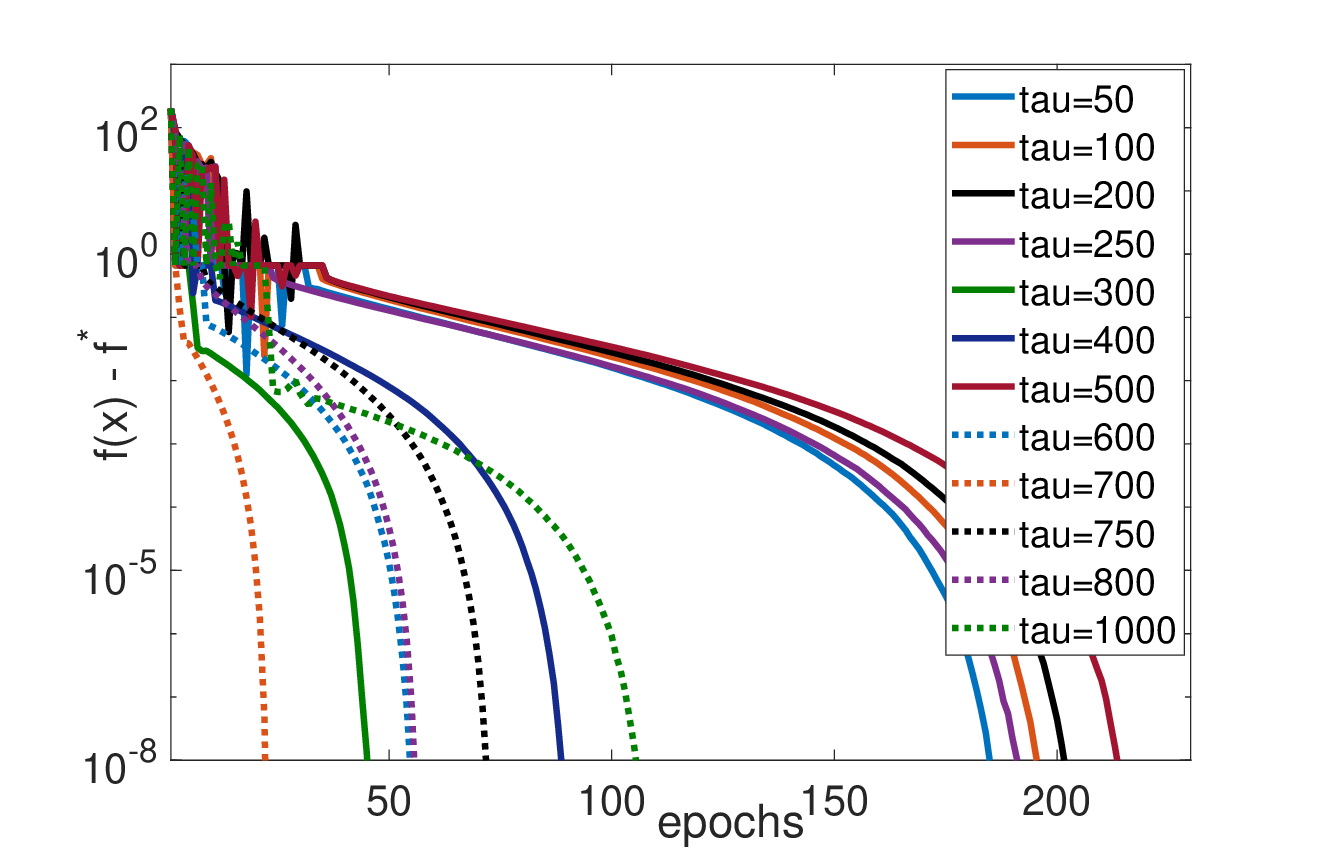}
	\hspace{-0.5cm}
	\includegraphics[height= 0.18\textheight, width=0.31\textheight]{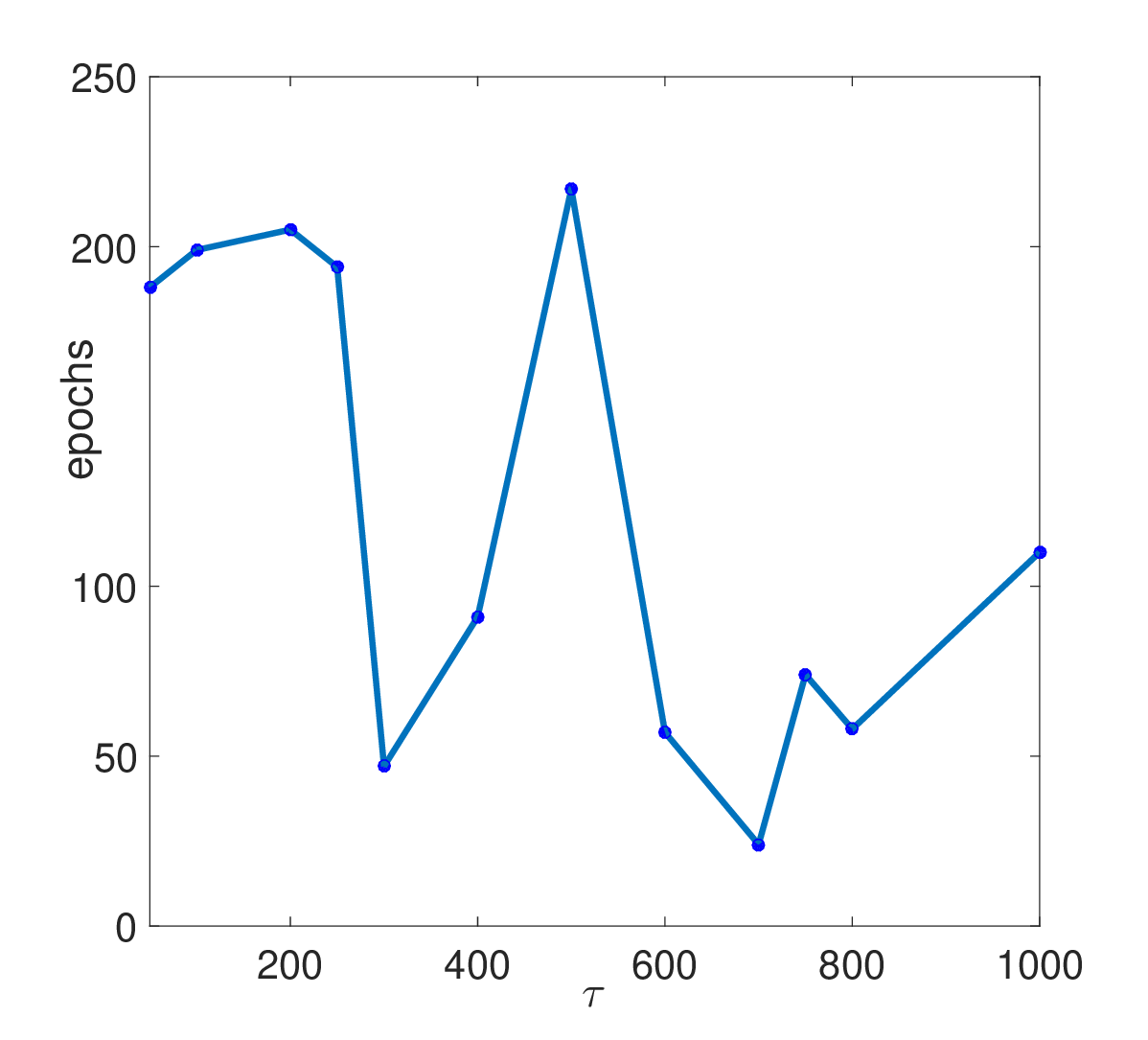}
	\caption{Behavior of SHOM for  different values of minibatch size $\tau$ ranging from $50$ to $1.000$ on  a8a dataset: $p=2$ (top)  and  $p=3$ (bottom).  }
	\label{fig:p2_tau_var}
\end{figure}

\noindent Further, we analyse the behavior of algorithm SHOM for different values of the minibatch size $\tau$ ranging from $50$ to $1.000$ on a8a dataset. The results are given in Figure 2 for $p=2$ (top) and $p=3$ (bottom).  We observe that there are two optimal regimes for the minibatch size $\tau^*$ and they coincide for the two variants of SHOM: $p=2$ and $p=3$, respectively. However, the number of epochs is smaller for $p=3$ than for $p=2$.  One can also observe that  large minibatch sizes  do not necessarily  improve~convergence.


	\subsection{Independent component analysis  problem}
	The independent component analysis (ICA) problem consists of finding the
	sources  $V=\left[\begin{array}{ll}v_{1}, & v_{2}, \ldots ,v_{r}\end{array}\right] \in \mathbb{R}^{r \times N}$ from some observed linear mixed set of signals $\hat{V} = [\hat{v}_1,\,\hat{v}_2,\ldots, \hat{v}_N] \in \mathbb{R}^{m \times N}$, i.e.,   $\hat{V} = AV,$  where $A$ is  the mixing matrix, which is assumed  unknown \cite{HyvKar:01}. The basic assumption is that the signals are statistically independent, so that a distribution of sum of independent signals with arbitrary distributions tends toward a Gaussian distribution. In other words, the purpose of the ICA is to find a demixing matrix $X = A^{\dagger}$ that separates the observed signals into a set of sources which are statistically independent.  Several techniques for measuring independence (non-Gaussianity) were proposed in the literature and posed as the following optimization problem, after some simplifications (see e.g., \cite{HyvKar:01} for more details):
	
	\vspace{-0.5cm}
	
	\begin{equation}
		\label{eq:ica_pb}
	\max_{\|x\| =1} \frac{1}{N}\sum_{i=1:N} \underbrace{\bar{\phi}(x^T\hat{v}_i)}_{=f_i(x)}.
	\end{equation}

\vspace{-0.3cm}

\noindent  One standard measure for independence is the kurtosis (the fourth order moment): $  \bar{\phi}(t) = t^4$. 	Although this measure of non-Gaussianity is theoretically good, it is sensitive to outliers. Therefore, a more practical choice is the approximation of negentropy that uses non-quadratic functions. Two classical choices are \cite{HyvKar:01}:
	$\bar{\phi}(t) =\frac{1}{\alpha} \ln \cosh (\alpha t), \;  \text{with}  \;  1\leq \alpha \leq 2 \quad \text{and} \quad \bar{\phi}(t)=-e^{-\frac{t^{2}}{2}}. $
	Note that  in all these cases  the optimization problem 	\eqref{eq:ica_pb}  is nonconvex, but for $p \in \{1,2\}$ the objective function $\bar \phi$ is  $p$	times differentiable and with the $p$th derivative  Lipschitz  over the bounded feasible  set  $\mathcal{X}=\{x: \|x\| =1 \}$.	We apply ICA framework to reduce the spectral  dimension of a hyperspectral image. A hyperspectral image $\hat{V}$ is a three dimensional hyperspectral data cube $n_1 \times n_2 \times m$, having   $ n_1 \times n_2$ pixels and $m$ is the number of spectral channels.  Denote $N= n_1\cdot n_2$  the number of pixels in each spectral channel.  Then, the hyperspectral image can be represented as a matrix: 
	\vspace*{-0.3cm}
	\begin{align*}
		{\hat{V} = \begin{bmatrix}
			\hat{v}_{11} & \hat{v}_{12} & \cdots &\hat{v}_{1N}   \\
			\cdots&\cdots&\cdots&\cdots \\
			\hat{v}_{m1} & \hat{v}_{m2} & \cdots &\hat{v}_{mN}  
		\end{bmatrix}} \in \mathbb{R}^{m \times N},
	\end{align*}
	where $\hat{v}_{ij}$ is the pixel $j$ on band $i$. We aim at reducing the number of bands in the hypespectral image from $m$ to $r \ll m$ using the negentropy-based  ICA framework,  replacing the expectation in the nonconvex problem 	\eqref{eq:ica_pb} with the empirical risk using the samples from $\hat{V}$. In our tests we use Salinas scene that has the spatial dimension $ 512 \times 217$ (i.e., the number of functions in the finite-sum objective is $N = 111.104$) and $204$ bands, see  \cite{himDB}. We compare SHOM with $p=1, 2$ and $\tau=150$ and the state of the art algorithm FastICA  developed specially for the ICA problem in \cite{HyvKar:01}.  Note that FastICA is a full batch gradient method with a scalar stepsize depending on the second derivatives of $\bar{\phi}$, see \cite{HyvKar:01} for details.  We choose $\bar{\phi}_i(t) = -e^{-t^2/2}$ in  \eqref{eq:ica_pb}, $r=1$, {$L_1^{f_i}= (1+\|\hat{v}_i\|^2) \|\hat{v}_i\|^2 \,  (M_1 \equiv M_1^i = 1.1 \max_{i=1:N}L_1^{f_i}) $   and $ L_{2}^{f_i} = \left(3\left\|\hat{v}_{i}\right\|+\left\|\hat{v}_{i}\right\|^{3}\right)\left\|\hat{v}_{i}\right\|^{3} \, (M_2 \equiv  M_2^i =1.1 \max_{i=1:N}L_2^{f_i} )$, for all $i=1:N$, respectively, and we solve the subproblem   \eqref{sim:eq10} using ADMM described in Section 6.1.}	 From Figure 3 (left) one can observe that SHOM for both  $p=1$ and $2$ is superior to FastICA algorithm, being able to achieve even high accuracy in a very small number of epochs. Finally, Figure 3 (right) display the sample band of Salinas dataset found with  FastICA algorithm  and  SHOM. We observe that the two images, corresponding to a single band, are  similar.        
	\begin{figure}[h] \label{fig:salina}
		\centering
		\includegraphics[height= 0.18\textheight, width=0.31\textheight]{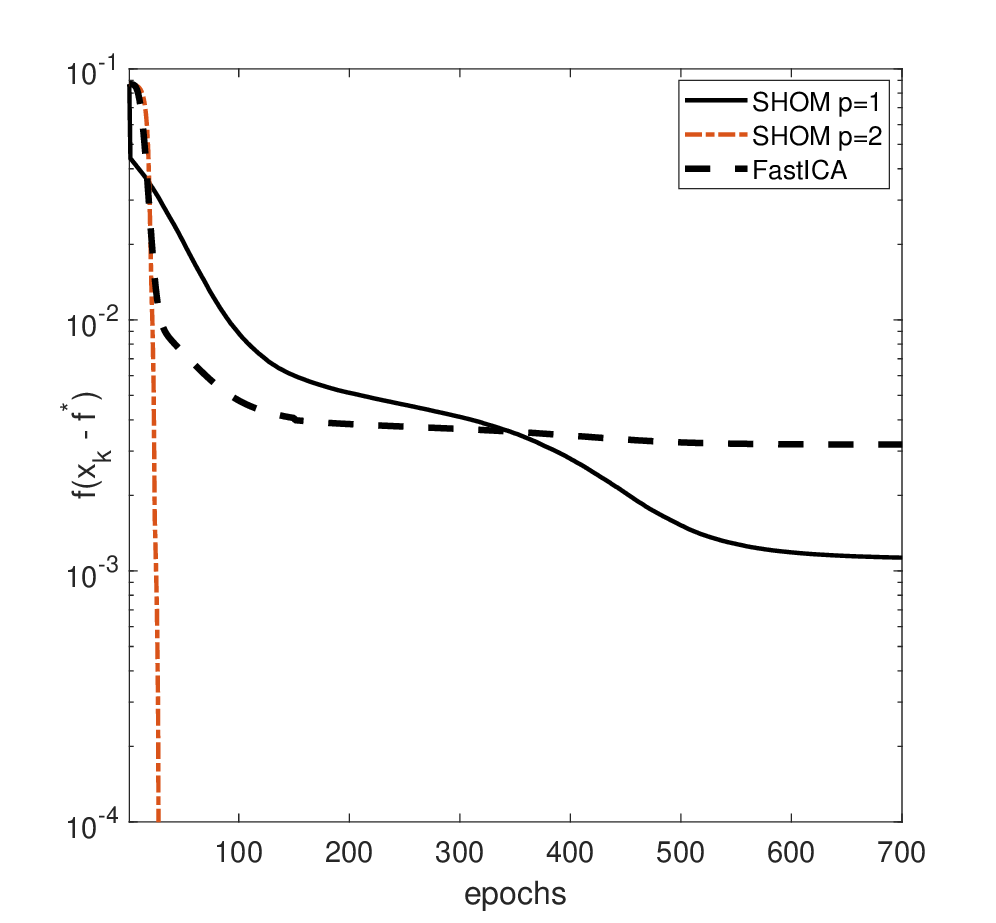}
		\includegraphics[height= 0.18\textheight, width=0.31\textheight]{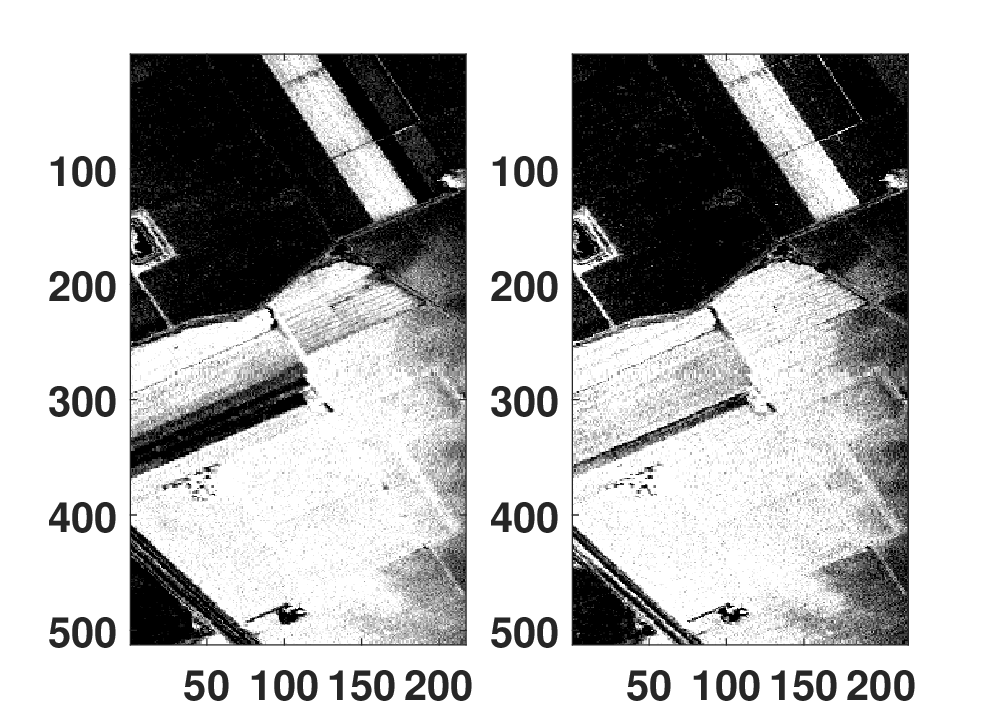}
		\caption{Left: behavior of SHOM and FastICA on ICA problem. Right: sample band of Salinas dataset corresponding to  FastICA and SHOM ($p=2$).}
	\end{figure}

\vspace{-0.3cm}

\section{Conclusions}
\vspace{-0.2cm}
\noindent In this paper, we have designed   a minibatch stochastic higher-order algorithm for minimizing finite-sum  optimization problems. Our method is  based on the minimization of a  higher-order upper bound approximation of the  finite-sum objective function.  We have presented convergence guarantees for nonconvex and convex optimization when the higher-order upper bounds approximate the objective function  up to an error  whose  derivative of a certain order  is Lipschitz continuous. More precisely, we have derived  asymptotic stationary point guarantees for nonconvex problems, and for objective functions  having the KL property or uniformly convex ones we have established local (super\textcolor{black}{/sub})linear  convergence rates.  Moreover,   unlike other higher-order algorithms, our method works with any  batch size.  Numerical simulations have also confirmed the efficiency of our algorithm. 

\vspace{-0.3cm}


\section*{Acknowledgements}
\vspace{-0.3cm}
The research leading to these results has received funding from the NO Grants 2014–2021 RO-NO-2019-0184, under project ELO-Hyp, contract no. 24/2020; UEFISCDI PN-III-P4-PCE-2021-0720, under project L2O-MOC, nr. 70/2022.



\vspace{-0.3cm}

\begin{thebibliography}{100}
	\bibitem{AgaKam:20}
	A. Agafonov, D. Kamzolov, P. Dvurechensky and A. Gasnikov,
		\emph{Inexact tensor methods and their application to stochastic convex optimization}, arXiv preprint: 2012.15636, 2020.
	
	\bibitem{BatCur:18}
	L. Bottou, F. Curtis  and J. Nocedal,  \textit{Optimization methods for large-scale machine learning},  SIAM Review,  60: 223--311, 2018. 
	
	\bibitem{BirGar:17}
	E.G. Birgin, J.L. Gardenghi, J.M.  Martinez, S.A. Santos and Ph.L. Toint, \textit{Worst-case evaluation complexity for unconstrained nonlinear optimization using high-order regularized models},  Mathematical Programming, 163(1-2): 359--368, 2017.
	
	\bibitem{BolByr:18}
	R. Bollapragada, R.H. Byrd and J. Nocedal, \textit{Exact and inexact subsampled
		Newton methods for optimization},  IMA Journal of Numerical Analysis, 39(2): 545-578, 2018.
	
	\bibitem{BolDan:07}
	J. Bolte, A. Daniilidis and A. Lewis, \textit{The Lojasiewicz inequality for nonsmooth subanalytic functions with applications to subgradient dynamical systems}, SIAM Journal on Optimization, 17: 1205--1223, 2007.
	
	\bibitem{BoyPar:11}
	S. Boyd, N. Parikh, E. Chu, Eric, B. Peleato and J. Eckstein, \emph{Distributed optimization and statistical learning via the alternating direction method of multipliers}, Foundations and Trends in Machine learning, 3(1): 1--122, 2011.
	
	\bibitem{cartis2017}
	C. Cartis, N. IM. Gould and P.L Toint, \emph{A concise second-order complexity analysis for unconstrained optimization using high-order regularized models}, Optimization Methods and Software, 35(2): 243--256, 2020.  
	
	\bibitem{Libsvm}
	C.C. Chang and C.J. Lin, \emph{LIBSVM: a library for support vector machines},  ACM Transactions on Intelligent Systems and Technology, 27(2): 1--27, 2011.
	
	\bibitem{DefBac:14}
	A. Defazio, F. Bach and S. Lacoste-Julien, \emph{SAGA: A fast incremental gradient method with support for non-strongly convex composite objectives},  Advances in Neural Information Processing Systems, 2014.

	
	\bibitem{DoiNes:2019}
	N. Doikov and Yu. Nesterov,  \emph{Local convergence of tensor methods}, Mathematical Programming, 193: 315--336, 2022.
	
	\bibitem{DoiNes:20}
	N. Doikov and Yu. Nesterov, \emph{ Inexact Tensor Methods with Dynamic Accuracies}. International Conference on Machine Learning, pp. 2577-2586, 2020.
	
	\bibitem{Gas:19}
	A. Gasnikov, P. Dvurechensky, E. Gorbunov, E. Vorontsova, 
		D. Selikhanovych,  C.A. Uribe, B.  Jiang, H. Wang,  S.  Zhang,  S. Bubeck and  Q. Jiang, \emph{Near optimal methods for minimizing convex functions with Lipschitz $p$th derivatives},  Conference on Learning Theory, 1392--1393, 2019.
	
	
	\bibitem{GooBen:16}
	I. Goodfellow, Y. Bengio and  A.  Courville, \emph{Deep learning}, MIT Press, 2016.
	
	
	\bibitem{HyvKar:01}
	A. Hyvarinen, J. Karhunen and E. Oja, \emph{Independent component analysis}, Wiley,   2001.
	
	\bibitem{JohZha:13}
	R. Johnson and T. Zhang, \emph{Accelerating stochastic gradient descent using predictive variance reduction}, Advances in Neural Information Processing Systems, 315--323, 2013.
	
	\bibitem{KovRic:19} 
	D. Kovalev, K. Mishchenko and P. Richtarik, \emph{Stochastic Newton and cubic Newton methods with simple local linear-quadratic rates},  Advances in Neural Information Processing Systems,  2019. 
	
	\bibitem{LucKoh:19}
	A. Lucchi and J. Kohler, \emph{A Stochastic tensor method for non-convex optimization},  arXiv preprint: 1911.10367, 2019. 
	
	\bibitem{MokEis:18}
	A. Mokhtari, M. Eisen and A. Ribeiro, \textit{ IQN: An incremental quasi-Newton
		method with local superlinear convergence rate},  SIAM J. on Optimization, 28(2): 1670--1698, 2018.
	
	\bibitem{MouBac:11}
	E.  Moulines and F. Bach, \textit{Non-asymptotic analysis of	stochastic approximation algorithms for machine learning},    Advances in Neural Information Processing System, 2011.
	
	\bibitem{Mairal:15}
	J. Mairal, \emph{Incremental majorization-minimization optimization with application to large-scale machine learning},  SIAM Journal on Optimization, 25(2): 829--855, 2015.
	
	
	
	\bibitem{Nes:19}
	Yu. Nesterov, \emph{Implementable tensor methods in unconstrained convex optimization},  Mathematical Programming,  186(1): 157--183, 2021.
	
	\bibitem{Nes:Inexat19}
	Yu. Nesterov, \emph{Inexact basic tensor methods for some classes of convex optimization problems},  Optimization Methods and Software,  doi: 10.1080/10556788.2020.1854252, 2020.
	
	\bibitem{Nes:13}
	Yu. Nesterov, \emph{Gradient methods for minimizing composite functions},  Mathematical Programming, 140: 125--161,  2013. 
	
	\bibitem{Nec:20}
	I. Necoara, \emph{General convergence analysis of stochastic first order methods for composite  optimization}, Journal of Optimization Theory and Applications, 189(1): 66--95, 2021.
	
	\bibitem{NecNed:11}
	I. Necoara, V. Nedelcu and I. Dumitrache, \textit{Parallel and distributed optimization methods for estimation and control in networks}, Journal of  Process Control, 21(5):756--766, 2011.
	
	\bibitem{NecLup:20}
	I. Necoara and D. Lupu, \emph{General higher-order majorization-minimization algorithms for (non)convex optimization}, arXiv preprint:  2010.13893, 2020.  
	
	\bibitem{NemJud:09}
	A. Nemirovski, A. Juditsky, G. Lan and A. Shapiro, \textit{Robust stochastic approximation approach to stochastic programming}, SIAM J. on Optimization, 19(4): 1574--1609, 2009.
	
	\bibitem{NesBor:06}
	Yu. Nesterov and B.T. Polyak, \emph{Cubic regularization of Newton method and its global performance},  Mathematical Programming, 108(1):  177--205,  2006.
	

	\bibitem{NguSch:17}
	L.M. Nguyen, J. Liu, K. Scheinberg and M. Takac, \emph{SARAH: A novel method for machine learning problems using stochastic recursive gradient}, International Conference on Machine Learning, 2017.

	

	\bibitem{RodNes:20}
	A. Rodomanov and Yu. Nesterov, \emph{Smoothness parameter of power of Euclidean norm}, Journal of Optimization Theory and Applications,  185(2): 303--326, 2020.
	
	\bibitem{RosVil:14}
	L. Rosasco, S. Villa and B.C. Vu, \textit{Convergence of stochastic
		proximal  gradient algorithm},  Applied Mathematics and Optimization,  82:  891--917, 2020.
	
	\bibitem{TriSte:18}
	N.  Tripuraneni, M. Stern, C. Jin, J. Regier and M.I.  Jordan,  \textit{Stochastic
		cubic regularization for fast nonconvex optimization},  Advances in Neural Information
	Processing Systems,  2899--2908, 2018.
	
	\bibitem{ZhoXu:18}
	D. Zhou, P. Xu and Q. Gu,  \textit{Stochastic variance-reduced cubic regularized
		Newton method},  International Conference on Machine Learning, 5985--5994, 2018.
	
	\bibitem{himDB}
	{http://www.ehu.eus/ccwintco/index.php/Hyperspectral-Remote-Sensing-Scenes.}
	\bibitem{ipopt}
	A. Wächter and L. T. Biegler \emph{On the Implementation of a Primal-Dual Interior Point Filter Line Search Algorithm for Large-Scale Nonlinear Programming}, Mathematical Programming 106(1): 25--57, 2006.
\end{thebibliography}
\end{document}